\documentclass[review,hidelinks,onefignum,onetabnum]{siamart220329}
\usepackage{stmaryrd}
\usepackage{amsmath, amssymb, amsfonts}

\usepackage{lipsum}
\usepackage{amsfonts}
\usepackage{graphicx}
\usepackage{epstopdf}
\usepackage{algorithmic}
\ifpdf
  \DeclareGraphicsExtensions{.eps,.pdf,.png,.jpg}
\else
  \DeclareGraphicsExtensions{.eps}
\fi

\newsiamremark{remark}{Remark}
\newsiamremark{hypothesis}{Hypothesis}
\crefname{hypothesis}{Hypothesis}{Hypotheses}
\newsiamthm{claim}{Claim}

\headers{An OSFV Method for Hyperbolic Conservation Laws}{X. Y. zhang, C. W. Cao, and P. Liang}

\title{An Oscillation-free Spectral Volume Method for Hyperbolic Conservation Laws \thanks{Submitted to the editors DATE.
\funding{This research is supported by National Natural Science Foundation of China under grants No. 11871026, 12271049, 11701038.}}}

\author{Xinyue Zhang\thanks{Laboratory of Mathematics and Complex Systems, School of Mathematical Sciences, Beijing Normal University, Beijing, China  
  (\email{zhang\_xinyue@mail.bnu.edu.cn}).}
 \and Waixiang Cao\thanks{Laboratory of Mathematics and Complex Systems, School of Mathematical Sciences, Beijing Normal University, Beijing, China  
 (\email{caowx@bnu.edu.cn}), corresponding author.}
 \and Liang Pan\thanks{Laboratory of Mathematics and Complex Systems, School of Mathematical Sciences, Beijing Normal University, Beijing, China  
 (\email{panliang@bnu.edu.cn}).}
}
\usepackage{amsopn}

\ifpdf
\hypersetup{
  pdftitle={An OSFV method for Hyperbolic Conservation Laws},
  pdfauthor={X.Y. Zhang, C.W. Cao, and P. Liang}
}
\fi

\begin{document}

\maketitle

\begin{abstract}
In this paper, an oscillation-free spectral volume (OFSV) method  is proposed and studied for the hyperbolic conservation laws.  
The numerical scheme is designed by introducing a damping term in the standard spectral volume method for the purpose of 
controlling spurious oscillations near discontinuities. Based on the construction of control volumes (CVs), two classes of  OFSV 
schemes are presented. A mathematical  proof is provided to show that the proposed OFSV is stable and has optimal convergence rate  
and some desired superconvergence properties when applied to the linear scalar equations.  Both analysis and numerical experiments  
indicate that the damping term would not destroy the order of accuracy of the original  SV scheme and can control the  oscillations  
discontinuities effectively. Numerical experiments are presented to demonstrate the accuracy and robustness of our scheme. 
\end{abstract}



\section{Introduction}

The development of high-order schemes for hyperbolic conservation laws has becomes extremely
demanding for computational fluid dynamics. In recent years, a great number  high-order numerical
schemes have been developed, including discontinuous Galerkin (DG) \cite{DG0,DG1,DG2},  
spectral volume (SV) \cite{SV1,SV2,SV3,SunWang2014}, spectral difference (SD) \cite{SD}, correction procedure 
using reconstruction (CPR) \cite{CPR}, essential non-oscillatory (ENO)
\cite{ENO1,ENO2}, weighted essential non-oscillatory (WENO)
\cite{WENO-Liu,WENO-JS,WENO-Z}, Hermite WENO (HWENO)
\cite{HWENO1,HWENO2,HWENO3} methods, etc.

One main difficulty for the high-order numerical schemes is the spurious oscillations near discontinuities, which lead to the nonlinear instability and eventual blow up of the codes.
Therefore, it is important to eliminate the oscillations near discontinuities and maintain the original high-order accuracy in the smooth regions. 
Many limiters have been studied in the literatures for DG methods.  The minmod total variation bounded (TVB) limiter \cite{DG1,TVB1,TVB2} was originally developed, 
which is a slope limiter using a technique from the finite volume method \cite{MUSCL}. One disadvantage of such limiter is that it may degrade accuracy in smooth regions.
The WENO reconstructions are also used as limiter for the DG methods \cite{WENOlimiter1,WENOlimiter2,WENOlimiter3}. In these methods, the WENO reconstructions were 
used to reconstruct the values at Gaussian quadrature points in the target cells, and rebuild the solution polynomials from the original cell 
average and the reconstructed values through a numerical integration.  Although the original order of accuracy can be kept,  the WENO limiters need a very large stencil, 
which is complicated to be implemented in multi-dimensions, especially for unstructured meshes.  
To use the compact stencils, the HWENO limiters are developed  \cite{HWENO1,HWENO2,HWENO3}, which reduce the stencil of the reconstruction by utilizing
both the cell averages and the spatial derivatives from the neighbors.  The role of limiters is to “limit” and ”preprocess” numerical  solution in the
“troubled cells”, which is often identified by a troubled cell indicator. Another method is to add artificial diffusion terms in
the weak formulation,  which should be carefully designed to  ensure entropy stability and suppress the oscillations essentially \cite{ariticaldiffusion}. 
 Recently, an oscillation-free DG (OFDG) method was developed for the scalar hyperbolic
conservation laws in \cite{OFDG-1} by introducing a damping term on
the classical DG scheme.  The damping term was carefully constructed, which controls spurious nonphysical oscillations near discontinuities and maintains uniform high-order
accuracy in smooth regions simultaneously. One advantage of the damping
term was that the damping technique is convenient for the
theoretical analysis, at least in the semi-discrete analysis,
including conservation, $L^2$-stability, optimal error estimates,
and even superconvergence. The OFDG has also been extended  to the
hyperbolic systems  \cite{OFDG-2} and well-balanced shallow water equations \cite{OFDG-3}.

The main purpose of the current work is to adopt the idea of damping term in \cite{OFDG-1} to the 
standard SV method for  system of hyperbolic equations. 
The SV method was originally formulated and later developed  for hyperbolic equations by Wang  and his colleagues \cite{SV1,SV2,SV3}, 
which might be regarded as a generalization of the classic Godunov finite volume method \cite{Riemann-exact,Riemann-appro}.
The SV method enjoys many excellent properties such as  high-order accuracy, compact stencils,
and geometrical flexibility (applicable for unstructured grids). In particular, since the SV method preserves conservation laws 
on more finer meshes, it might have a higher resolution for discontinuities  than other high order methods (see \cite{SunWang2014}).  
During the past decades, 
the SV method has been rapidly applied to solving various PDEs such as  the shallow
water wave equation \cite{SVShallow},  Navier Stokes equation
\cite{SVNS1,SVNS2}, and electromagnetic field \cite{SV1}, and so on. 
A mathematical analysis   in terms of the $L^2$ stability, accuracy,  error estimates, and superconvergence of the 
SV method was conducted  
in \cite{SV-Cao} under the framework of the  Petrov-Galerkin method.
It was also proved in \cite{SV-Cao} that a special class of SV scheme is exactly the same as the upwind DG schemes
when applied to linear constant hyperbolic equations.

By introducing the idea of damping term into the SV scheme,  a new OFSV method is proposed in this article. 
The OFSV method inherits both advantages of the standard SV method and the damping term. Specifically, 
In  one hand,  properties  such as high order accuracy, local conservation,  $h$-$p$ adaptivity, flexibility of handling unstructured meshes can still be preserved for OFSV method. 
On the other hand,   the OFSV method could effectively 
 control spurious nonphysical
oscillations near discontinuities and maintaining uniform high-order
accuracy in smooth regions,  without any use of limiters. 
Furthermore, note  that the SV method has larger CFL condition number than the DG method (see, e.g., 
\cite{SVDG-2}), which suggests that the proposed OFSV method might have a looser stability condition than the 
counterpart OFDG method.

The rest of this paper is organized as follows. In Section 2,
the OFSV method  for systems of hyperbolic
conservation laws is introduced. In Section 3,  a theoretical analysis is provided 
to show that proposed OFSV method is stable and has optimal convergence rate and superconvergnce 
property similar to the standard SV method, when applied to
for one-dimensional constant-coefficient linear scalar problem. 
In Section 4, we present some numerical examples to show the efficiency and robustness of our algorithm. 
 Finally, some concluding remarks are presented in Section 5.

\section{Oscillation-free SV method}

We consider the OFSV method for the following hyperbolic system
\begin{equation}\label{conlaws2d}
\left\{\begin{array}{l}
\boldsymbol{U}_t+\displaystyle\sum_{i=1}^{d} \boldsymbol{F}_i(\boldsymbol{U})_{x_i}=0, \quad(\boldsymbol{x}, t) \in \Omega \times(0, T], \\
\boldsymbol{U}(\boldsymbol{x}, 0)=\boldsymbol{U}_0(\boldsymbol{x}), \quad\boldsymbol{x} \in \Omega,
\end{array}\right.
\end{equation}
where $ \boldsymbol{x}=(x_1,\ldots, x_d)^T,$
$\boldsymbol{U}=(U_1,
\ldots, U_m)^T$, $\boldsymbol{F}_i(\boldsymbol{U})=(F_{i,
1}(\boldsymbol{U}), \ldots, F_{i, m}(\boldsymbol{U}))^T,  1\le i\le d$, and $\Omega $ is an open bounded domain in $\mathbb{R}^d$. 
  A large amount of physical models can be rewritten into the form of \eqref{conlaws2d}, such as the 
Euler equations for the  two-dimensional inviscid compressible flow, where $\boldsymbol{U}=\left(\rho,\rho U, \rho V,\rho E \right)$, 
$\boldsymbol{F}_1(\boldsymbol{U})=(\rho U,  \rho U^2+p,  \rho UV, (\rho E+p) U)$,
$\boldsymbol{F}_2(\boldsymbol{U})=(\rho V,  \rho UV,  \rho V^2+p, (\rho E+p) V)$
with $\rho E= 1/2\rho (U^2+V^2)+p/(\gamma-1)$ for the two-dimensional ideal gas.

  \begin{figure}[!h]
    \centering
    \includegraphics[width=0.3\textwidth]{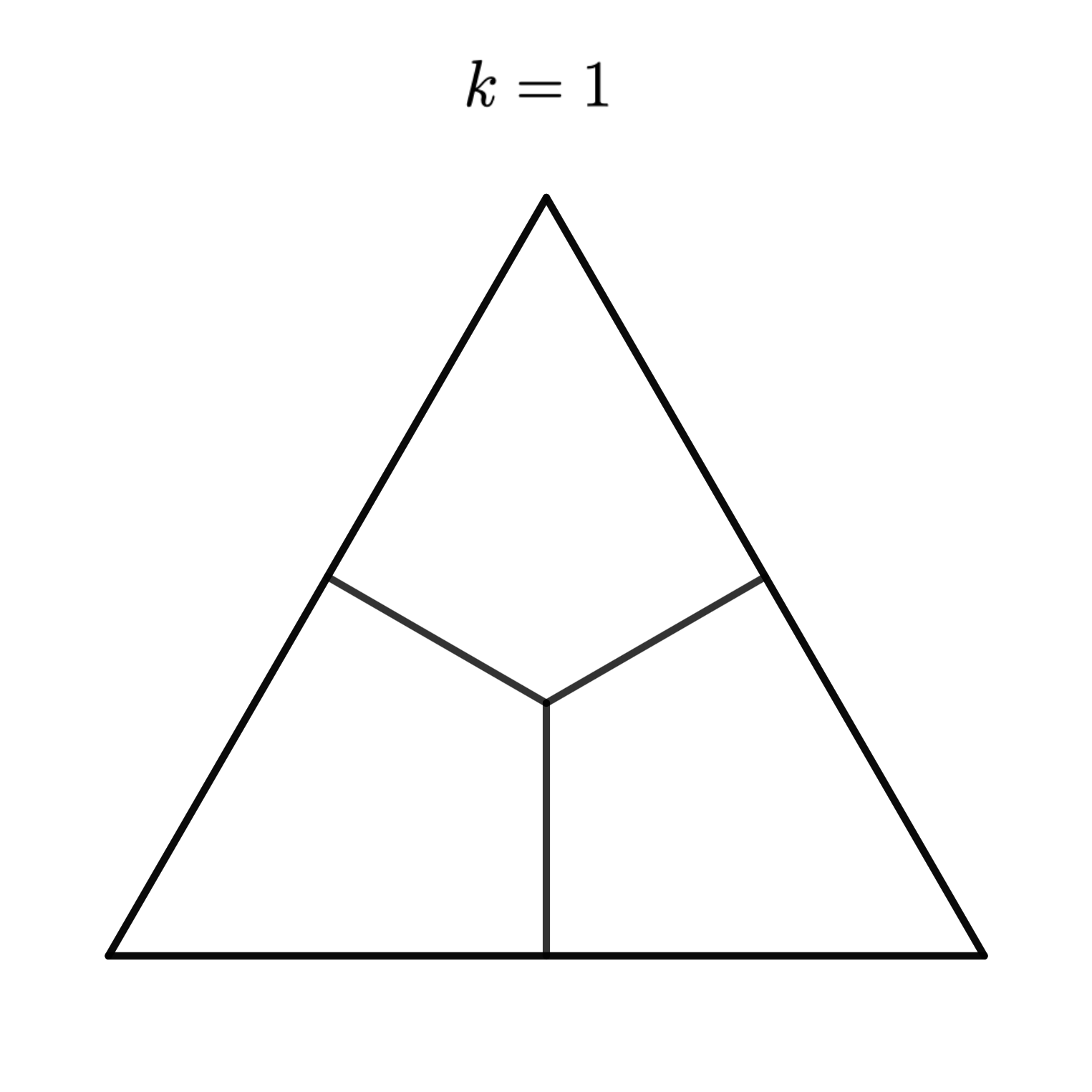}
    \includegraphics[width=0.3\textwidth]{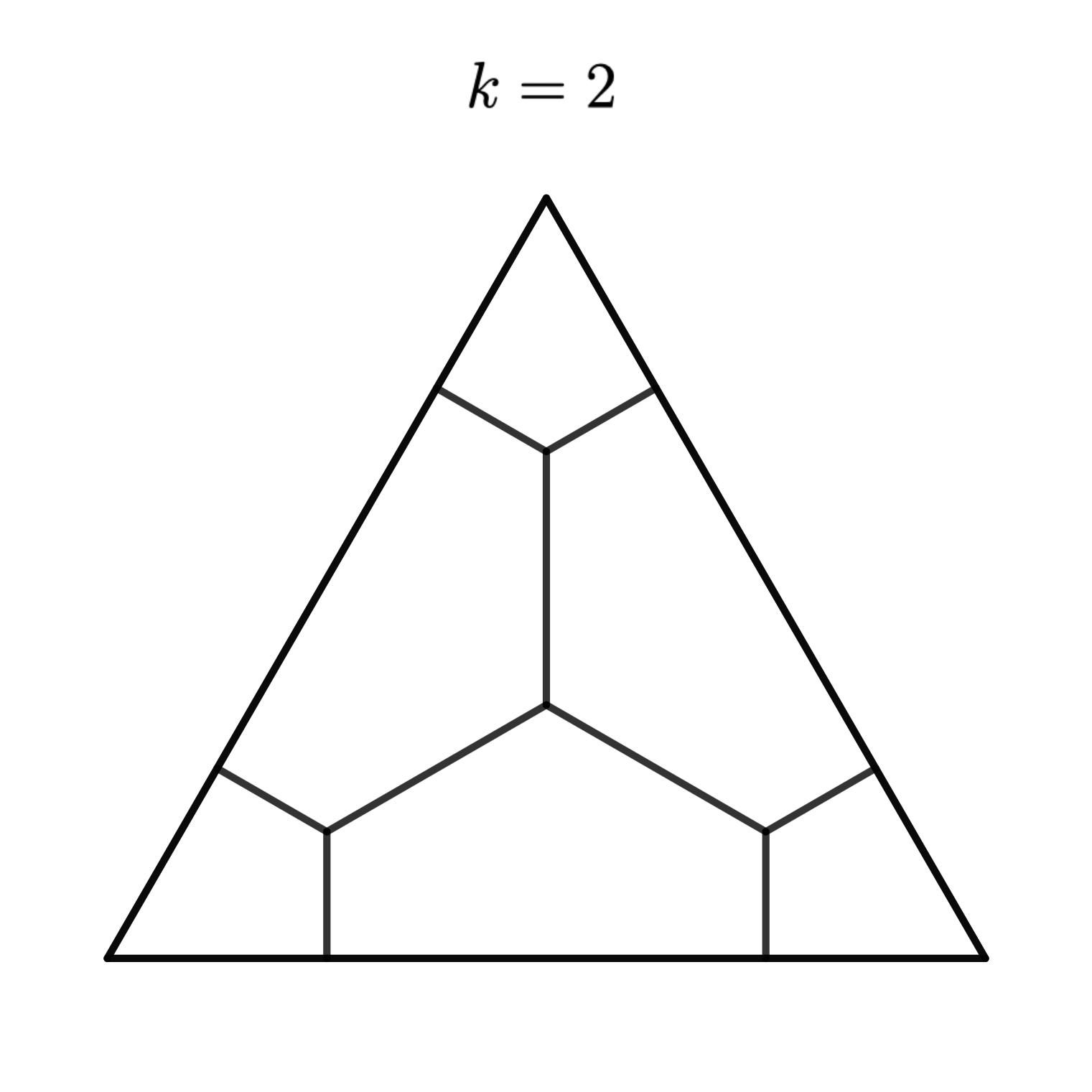}\\
    \includegraphics[width=0.3\textwidth]{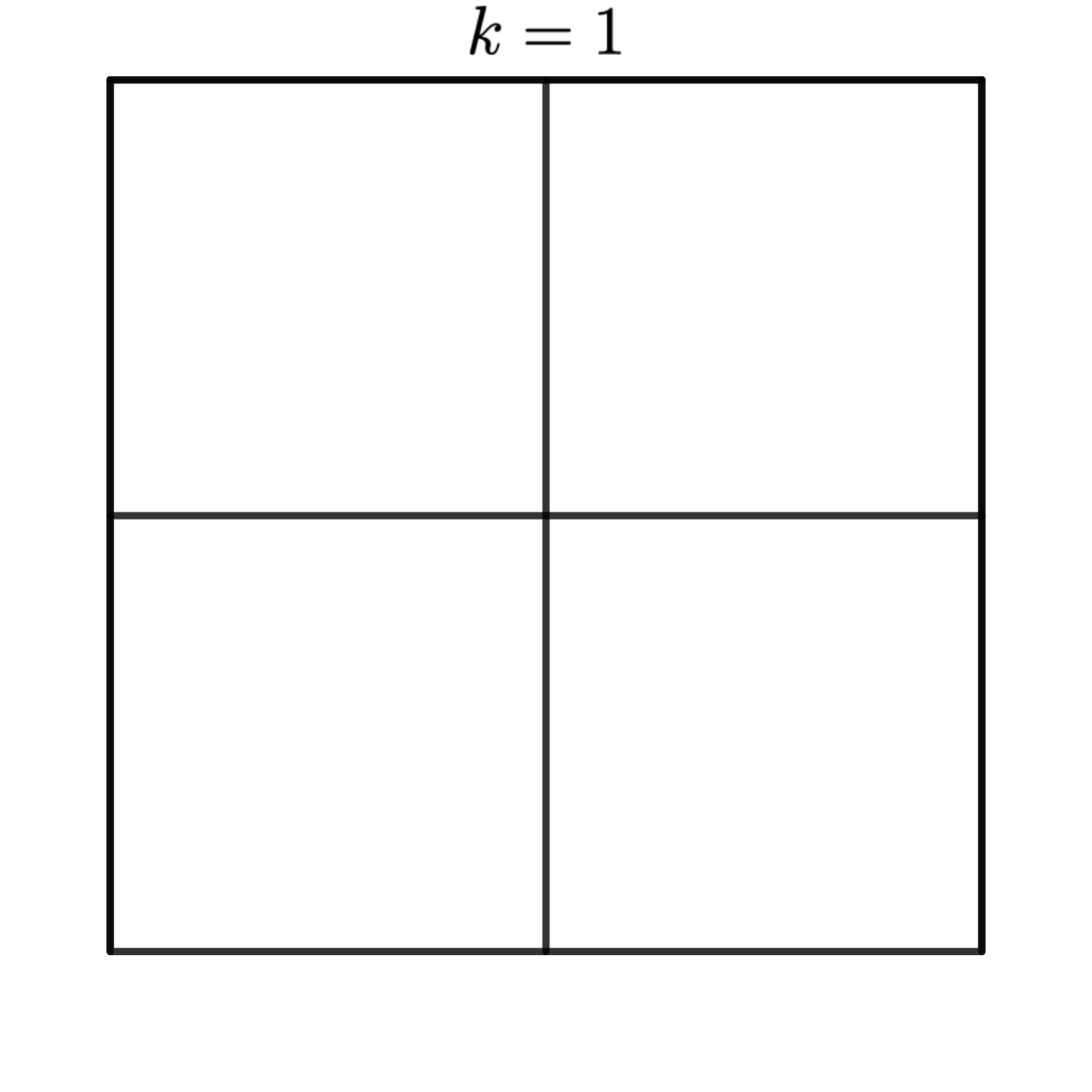}
    \includegraphics[width=0.3\textwidth]{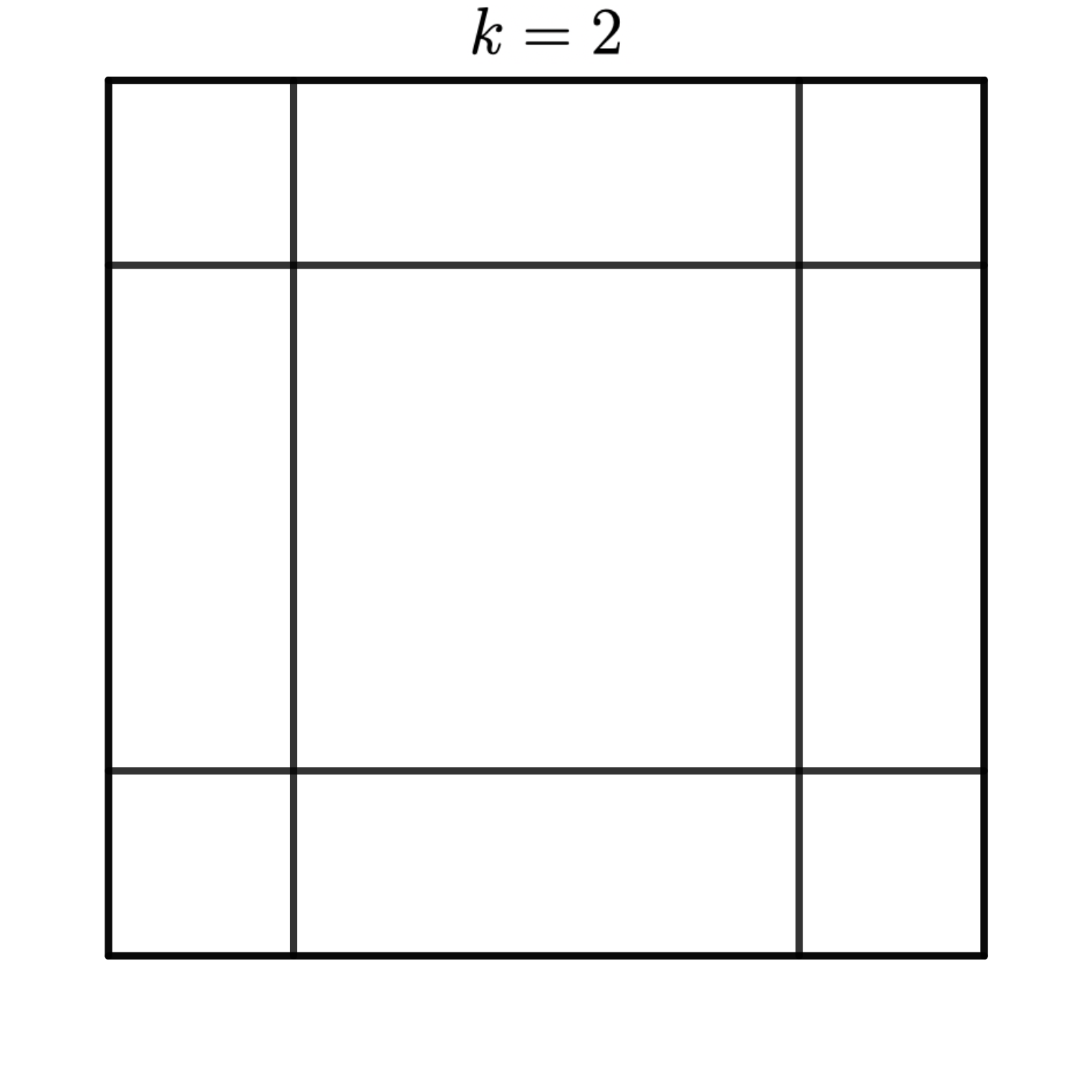}
    \caption{\label{CVs}  Partitions of a triangular SV for  $\mathcal{P}_k$ and a rectangular SV 
    for $\mathcal{Q}_k$ in 2D with $k=1,2$.}
    \end{figure}

To introduce the OFSV scheme, we first suppose that 
there exists a partition  $\mathcal{T}_h$ of $\Omega$ and  $\mathcal{T}_h$  is shape regular, 
 i.e., there exists a constant  $c>0$ such that 
\begin{equation*}
h \leq c h_{\tau},\ {\rm with}\ h = \displaystyle\max_{\tau \in \mathcal{T}_h} h_{\tau}, h_{\tau}=\operatorname{diam} \tau, ~ \tau\in \mathcal{T}_h.
\end{equation*}  
The (discontinuous) finite element space is defined as follows:
\begin{equation*}\label{trailespace}
     \mathcal{U}_{h}=\mathcal{U}_{h}^{k}=\{v \in L^{2}(\Omega):\left.v\right|_{\tau} \in  \mathcal{P}_k\  {\rm or}\  \mathcal{Q}_{k}, \tau \in \mathcal{T}_{h}\},
\end{equation*}
where $\mathcal{P}_k$ and $\mathcal{Q}_{k}$  denote the finite element space  of polynomials  with degree not greater than $k$ and 
  the $bi$-$k$ tensor product polynomial space, respectively. 
In the SV method, a simplex grid element $\tau$, called a spectral volume (SV), is futher divided into non-overlapping sub-elements $\{\tau_j^*\}_j$, 
named control volumes (CVs). The number of CVs hinges on the cardinality of the polynomials.  
To illustrate,  in 2D case, a SV is segmented into $N$ CVs with $N=(k+1)(k+2)/2$ for $\mathcal{P}_k$ elements, 
and $N=(k+1)^2$ for $\mathcal{Q}_k$ elements. Figure.1 shows partitions of a triangular SV and a rectangular SV  for linear and quadratic elements in 2D.

    Denote by ${\mathcal{T}^*_h}$ the dual partition of $ {\mathcal{T}_h}$, i.e., ${\mathcal{T}^*_h}=\{\tau_j^*: \tau\in {{\mathcal T}_h}\}.$ 
    The OFSV scheme for  equation \eqref{conlaws2d} read as:
    Seek $\boldsymbol{U}_h(\cdot, t) \in\left[\mathcal{U}_h^k\right]^m$ such that for all $\tau_j^*\in {\mathcal{T}^*_h}$,
   
    \begin{equation}\label{OFSV2D}
    \begin{aligned}
    \int_{\tau_j^*}\left(\boldsymbol{U}_h\right)_{t} \text{d}
    \boldsymbol{x}=\int_{\partial\tau_j^*}
    \widehat{\boldsymbol{F}} \cdot \boldsymbol{n}  \text{d}\Gamma
    -\sum_{l=0}^{k} \frac{\sigma_{\tau}^{l}}{h_{\tau}}
    \int_{\tau_j^*}
    \left(\boldsymbol{U}_h-P_{h}^{l-1}\boldsymbol{U}_h\right)\text{d}
    \boldsymbol{x}, 
    \end{aligned}
    \end{equation}
    where  $\widehat{\boldsymbol{F}}=(\widehat{\boldsymbol{F}}_1,\cdots,\widehat{\boldsymbol{F}}_d)$, and $\boldsymbol{n} = (n_1,\cdots,n_d)$ is the unit outward normal
    with respect to $\tau_j^*$. The Gaussian quadrature is used for the integration of numerical fluxes.  For the CV interfaces inside the SV, the smooth Euler flux
    is used. For the SV interfaces, the numerical fluxes are provided by  Riemann fluxes to deal with discontinuities,
     such as  Lax-Friedrichs flux and HLLC flux \cite{Riemann-appro}. 
     $P_h^l$ is the standard $L^2$ projection into
    $[\mathcal{U}_h^l]^m, l \geq 0$  with $P_h^{-1}=P_h^0$, and 
     $\sigma_{\tau}^l$ are damping coefficients, which are adopted to control spurious oscillations near the discontinuities.    
     The damping coefficients $\sigma_{\tau}^l$ are given as follows 
    \begin{equation}\label{xigema_ij}
        \sigma_{\tau}^l=\frac{2(2 l+1)}{(2 k-1)} \frac{h^l}{l !} \max _{1 \leq s \leq m} \sum_{|\boldsymbol{\alpha}|=l}\left(\frac{1}{N_e} 
        \sum_{\boldsymbol{v} \in \partial \tau}\left(\left.\llbracket \partial^{\boldsymbol\alpha} V_s\rrbracket\right|_{\boldsymbol{v}}\right)^2\right)^{\frac{1}{2}}, 
    \end{equation}
    where $N_e$ is the number of edges of the element $\tau$,   $\boldsymbol{v} \in \tau$ are the vertices of
    $\tau$,  $\left.\llbracket w \rrbracket\right|_{\boldsymbol{v}}$  denotes the jump of  $w$ on the vertex $\boldsymbol{v}$,  $\boldsymbol{\alpha}$ is the multi-index of order $|\boldsymbol{\alpha}|=\alpha_1+\cdots+\alpha_d,$
    and $\partial^\alpha w$ is defined as
    \begin{equation*}
    \partial^{\boldsymbol{{\boldsymbol\alpha}}} \boldsymbol{V}
    =\frac{\partial^{|\boldsymbol{{\boldsymbol\alpha}}|} \boldsymbol{V}}{\partial x_1^{\alpha_1}\cdots \partial x_d^{\alpha_d}}
    =\partial_{x_1}^{\alpha_1}\cdots  \partial_{x_d}^{\alpha_d} \boldsymbol{V}.
    \end{equation*}
    Here  the variables $\partial_x^l \boldsymbol{V}=\left(\partial_x^l V_1, \ldots, \partial_x^l V_m\right)^T$ 
    are given by $\partial_x^l \boldsymbol{V}=\boldsymbol{R}^{-1} \partial_x^l \boldsymbol{U}_h$, 
    $\boldsymbol{R}$ is the matrix corresponding to right eigenvectors of Jacobian matrix 
    $\partial \boldsymbol{F} /\partial \boldsymbol{U}((\overline{\boldsymbol{U}}_h)_{i+1/2})$ 
    and $\overline{(\cdot)}_{i+1/2}$ stands for the Roe average of variables at both side of point. 
    The numerical scheme \eqref{OFSV2D}
     reduces to the  standard SV method with $ \sigma_{\tau}^l=0$.

\begin{remark} 
The introduction of the damping term originates  from the OFDG method for hyperbolic conservation laws in \cite{OFDG-1}. 
Usually, higher order terms are  treated as sources of non-physical spurious oscillation. 
Therefore, the basic idea is to fix high order term by constructing a serious coefficients, 
which are small enough in the smooth region to guarantee the high order accuracy and 
sufficiently large near discontinuities to control spurious oscillations.  A natural way to establish 
the coefficients is adopting the jump of SVs vertexes, but it is not unique. More studies still need to be investigated in the  future research. 
\end{remark}

    \section{Analysis for constant coefficient hyperbolic problems}

    This section is dedicated to the analysis of the OFSV method, where the stability, convergence and 
    accuracy are discussed.  For simplify and clarity,  we focus our attention on the following 1D constant-coefficient linear advection
    equation
    \begin{equation}\label{1dlinear}
        \begin{cases}
            \partial_t u+\partial_x u=0, &(x,t) \in [a, b]\times(0,T],\\
            u(x, 0)=u_0(x), & x \in(a, b),
        \end{cases}
    \end{equation}
     with  the periodic boundary condition $u(a, t)=u(b, t)$ or the inflow boundary condition $u(0, t)=g(t)$. 
     The numerical fluxes are taken as upwind fluxes. 
     We would like to point out that 
    the same argument can be applied to the ${\mathcal Q}_k$ element for 
    2D linear scalar equations or systems, with some tedious  calculationas. 
    
    
    \color{black}
    
    \subsection{OFSV Method as a Petrov Galerkin Method}
    Let 
    \[
       {\mathcal T}_h=\{\tau_i: \tau_i=(x_{i-\frac 12}, x_{i+\frac 12}), \ 1\le i\le N\}. 
    \]
     Assume that the $i$-th SV $\tau_i$ is partitioned into $k+1$ CVs
     with $k$ points $x_{i,j}, 1\le j\le k$, denoted by $\tau_{i,j}, j=0,\cdots, k$ with 
     $\tau_{i,j}=(x_{i,j}, x_{i,j+1})$. Here $x_{i,0}=x_{i-\frac 12}, x_{i,k+1}=x_{i+\frac 12}$. 
     Then the dual partition 
     ${\mathcal T}_h^*$ can be expressed by 
     \[
        {\mathcal T}^*_h=\{\tau_{i,j}: \tau_{i,j}=(x_{i,j}, x_{i,j+1}), 1\le i\le N, 0\le j\le k\}. 
     \]
     Noticing that different choices of the dual partitions (i.e., $x_{i,j}$) leads to different SV scheme, which may have effect on 
     the stability of the numerical scheme.  Furthermore, as pointed out in \cite{SV-Cao},  if $x_{i,j}$ are taken as right Radau points, 
    the SV scheme is exactly the same as the standard upwind DG scheme. 
     In this paper, two classes of SV schemes 
      will be analyzed theoretically:  1) $x_{i,j}, 1\le j\le k$ are chosen as the Gauss points; 2) 
      $x_{i,j}, 1\le j\le k$  are taken as the interior right Radau points, (i.e., zeros of $L_{i,k+1}-L_{i,k}$  except the point $x_{i+\frac 12}$).
    Here $L_{i,k}$ denotes
      the Legendre polynomial of degree $k$ in  $\tau_i$.  The corresponding two  SV methods are separately referred as 
     the Gauss Legendre spectral volume (OF-LSV) method and the right Radau spectral volume (OF-RRSV) method.  
    
     To study the stability of OF-LSV and OF-RRSV, we first rewrite the SV scheme into its equivalent Petrov-Galerkin method. 
     Define the piecewise
     constant function space as
     \begin{equation*}
     \mathcal{V}_h=\left\{w^*: w^* \mid_{\tau_{i, j}} \in \mathcal{P}_0, 1 \leq i \leq N, 0 \leq j \leq k\right\}.
     \end{equation*}
     Obviously, any function $w^*=w^*(x, t) \in \mathcal{V}_h$ has the following formulation 
     $$
     w^*(x, t)=\sum_{i=1}^N \sum_{j=0}^k w_{i, j}^* \chi_{\tau_{i, j}}(x),
     $$
     where $w_{i, j}^*=w_{i, j}^*(t)$ are coefficients as functions of the variable $t$,  and $\chi_A, A \subset[a, b]$ is the characteristic function valuated $1$ in $A$ and $0$ otherwise. Let 
     $$
     \mathcal{H}_h=\left\{v: v|_{\tau_i} \in H^1, 1 \leq i \leq N\right\}. 
     $$
     Denote by $a_h(\cdot, \cdot)$ the bilinear form defined on $\mathcal{H}_h \times \mathcal{V}_h$, i.e., 
     \begin{equation}\label{bilinear:0}
         a_h\left(v, w^*\right)=\sum_{i=1}^N a_{h_i}\left(v, w^*\right), \quad v \in \mathcal{H}_h, w^* \in \mathcal{V}_h,
     \end{equation}
     with
     \begin{equation}\label{bilinear:1}
     a_{h_i}\left(v, w^*\right):=\sum_{j=0}^k w_{i, j}^*\big(\int_{x_{i, j}}^{x_{i, j+1}} v_td x+v^-_{i, j+1}-v^-_{i, j}
     + \displaystyle\sum^k_{l=0} \frac{\sigma_{i}^{l}}{h_{i}} \int_{x_{i, j}}^{x_{i, j+1}}(v_{h}-P_{h}^{l-1} v_{h})  d x\big),
     \end{equation}
     with $\sigma_{i}^{l} = \sigma_{{\tau}_i}^{l}$ given by \eqref{xigema_ij}. 
     Here $v^-_{i,j}$ denotes the right limit of $v$ at the point $x_{i,j}$. 
     
     Recalling the OFSV scheme in \eqref{OFSV2D}, 
     the OFSV method for \eqref{1dlinear} is to find a 
     $u_h \in \mathcal{U}_h$  such that 
     \begin{equation}\label{Gar}
         a_h\left(u_h, w^*\right)=0, \quad \forall w^* \in \mathcal{V}_h .
     \end{equation}
     Conversely, if $u_h \in \mathcal{U}_h$ satisfies the equation \eqref{Gar}, then
     $u_h$  is the solution of  \eqref{OFSV2D}, by
     choosing $w^*= \chi_{{\tau}_{i,j}} $.  In other words, the OFSV method
      is equivalent to the  numerical method \eqref{Gar}. 
     
     We end with this subsection the discussion about the bilinear form between the OFSV method and the 
     standard SV method. 
     Define  the bilinear form $a_{h}^{SV}(\cdot,\cdot)$ of the standard SV method by  (see, e.g., \cite{SV-Cao})
     \begin{equation*}
       a_h^{SV}(v,w)=\sum_{i=1}^Na_{h_i}^{SV}\left(v, w^{*}\right),\  {\rm with}\ 
       a_{h_i}^{SV}\left(v, w^{*}\right)=\sum_{j=0}^k w_{i, j}^* (\int_{x_{i, j}}^{x_{i, j+1}} v_td x+v^-_{i, j+1}-v^-_{i, j}). 
      \end{equation*}
       In light of \eqref{bilinear:0}-\eqref{bilinear:1},  there holds 
       \begin{equation*}\label{OFSV-SV}
        \begin{split}
        a_{h}\left(v, w^*\right)=a_{h}^{SV} (v,w^*)+ \sum^N_{i=1}\sum^k_{l=0} \frac{\sigma_{i}^{l}}{h_{i}} \left(v-P_{h}^{l-1} v, w^* \right).
        \end{split}
    \end{equation*}

    \subsection{$L^2$ stability}
      We first recall a  special transformation the trial space to the test space
      $\mathcal{F}: \mathcal{U}_h \rightarrow \mathcal{V}_h$, which is of great importance in our stability analysis. 
    For all $w \in \mathcal{U}_h$, let 
    \begin{equation*}
    \mathcal{F}  w=w^*:=\sum_{i=1}^N \sum_{j=0}^k w_{i, j}^*(t) \chi_{{\tau}_{i, j}}(x) \in \mathcal{V}_h,
    \end{equation*}
    where the coefficients $w_{i, j}^*$ are given as
    \begin{equation*}\label{w}
    w_{i, 0}^*=w_{i-\frac{1}{2}}^{+}, \quad w_{i, j}^*-w_{i, j-1}^*=A_{i, j} w_x\left(x_{i, j}\right),  \ 1\le j\le k.
    \end{equation*}
     Here $A_{i, j}$  denotes the 
    numerical quadrature weights corresponding the points $x_{i,j}$ in $\tau_i$.  
    Given any function $v\in {\mathcal H}_h$, 
    denote by $ R_i(v)$  the numerical quadrature error 
    between the exact integral 
    and its numerical quadrature $Q_k(v)$ in $\tau_i$, i.e., 
    \[
       R_i(v)=\int_{\tau_i} v dx-Q_k(v), \ \ {\rm with}\ Q_k(v)= \sum_{i=1}^{k+1} A_{i,j}v(x_{i,j}).
    \]
    Define 
    \[
       (v,w)_i=\int_{\tau_i} (vw) dx,\ (v,w)=\sum_{i=1}^N (v,w)_i,\ \ D_x^{-1} v=\int_{a}^x v dx. 
    \]
    It has been proved in  \cite{SV-Cao} that 
    \begin{equation}\label{integelerror}
      (v,w^*)_i- (v,w)_i=R_i\left(w_x D_x^{-1} v\right),\ \ \forall w\in {\mathcal U}_h, w^*={\mathcal F}w. 
    \end{equation}
     Due to the identity \eqref{integelerror},  the bilinear form of the OFSV can be rewritten into 
        \begin{equation}\label{OFSV-gar}
            a_{h_i}\left(v, w^*\right)=a_{h_i}^{SV}\left(v, w^*\right) + \sum^k_{l=0} \frac{\sigma_{i}^{l}}{h_{i}} \left(v-P_{h}^{l-1} v, w\right)_i+R_i \\\big(w_x\sum^k_{l=0} \frac{\sigma_{i}^{l}}{h_{i}}D^{-1}_x
    (v-P_{h}^{l-1} v)\big).
        \end{equation}
        Now we are ready to present the $L^2$ stability  for both OF-LSV and OF-RRSV method. 
        \begin{theorem}
             Suppose  $u_h$ is the solution of \eqref{Gar}. Then for both OF-LSV and OF-RRSV, 
             \begin{equation}\label{stabilityOFLSV}
              \|u_h(\cdot,t)\|_0 \lesssim \|u_0\|_0,\ \ \forall t>0. 
              \end{equation}
              \end{theorem}
              \begin{proof}
                Since the $k$-point Gauss numerical quadrature and the $(k+1)$-point right Radau numerical quadrature 
                is separately exact for polynomials of degree $2k-1$ and $2k$, there holds 
                $R_i(v)=0, v\in\mathcal P_{2k-1}$ for OF-LSV and 
                $R_i(v)=0, v\in\mathcal P_{2k}$ for OF-RRSV. 
               Then taking $v=w\in \mathcal U_h$ in   \eqref{OFSV-gar} and using the orthogonality property of $P_{h}^{l}$
               yields 
              \begin{equation}\label{mid2}
                          a_{h_i}\left(v, v^*\right)
                           = a_{h_i}^{SV}\left(v, v^*\right) + \sum^k_{l=0} \frac{\sigma_{i}^{l}}{h_{i}} \Vert v-P_{h}^{l-1} v\Vert_{0,\tau_i}^2 + R_i (v_x (\displaystyle\sum^k_{l=0}\frac{\sigma_{i}^{l}}{h_{i}} D_x^{-1}v)). 
              \end{equation}
              Noticing that $v_xD_x^{-1}v|_{\tau_i}\in\mathcal P_{2k}$ for any $v\in\mathcal U_h$, we have 
              $R_i (v_x D_x^{-1}v)=0$ for OF-RRSV. 
              As for OF-LSV,  by using the error of Gauss-Legendre quadrature, there exists a  $\xi_i\in\tau_i$ such that 
              \begin{equation}\label{mid3}
                  R_i(v_x D_x^{-1}v) = c_l^k (\frac{h_{i}} 2)^{2k+1}\partial_x^{2k} (v_xD_x^{-1}v)^{(2k)}(\xi_i)
                  =c_l^k  (\frac{h_{i}} 2)^{2k} C_{2k}^{k-1}(\partial_x^k v(\xi_i))^2 \geq 0, 
              \end{equation}
              where $c_l^k = \frac{2^{2k+1}(k!)^4}{(2k+1)[(2k)!]^3}$.   Then we conclude from \eqref{integelerror} and the inverse inequality that 
              \begin{equation}\label{equi:1}
                 (v,v)\le  (v,v^*)\lesssim (v,v),\ \ \forall v\in\mathcal U_h. 
              \end{equation}
                On the other hand,  substituting  \eqref{mid3} into \eqref{mid2} and summing up all $i$ from 
              $1$ to $N$ gives 
              \begin{equation}\label{LSVstability}
                          a_{h}\left(v, v^*\right)
                           \ge a_{h}^{SV}\left(v, v^*\right) + \sum_{i=1}^N\sum^k_{l=0} \frac{\sigma_{i}^{l}}{h_{i}} \Vert v-P_{h}^{l-1} v\Vert_{0,{\tau}_i}^2.
               \end{equation}
                Note that (see, e.g., \cite{SV-Cao}) 
              \begin{equation}\label{LSVstability:0}
                  a_{h}^{SV}\left(v, v^*\right)=(v_t,v^*)+\frac{1}{2}\sum_{i=1}^N \left.\llbracket v \rrbracket\right|^2_{i+\frac 12} +
              \frac{1}{2}(v_{N+\frac{1}{2}}^{-})^2-\frac 12 (v_{\frac{1}{2}}^{-})^2.
              \end{equation}
              By  taking 
               $v = u_h$ in \eqref{LSVstability} and using \eqref{LSVstability:0} and the $L^2$ equivalence \eqref{equi:1}, then  \eqref{stabilityOFLSV} 
               follows. 
              \end{proof}

               \subsection{Optimal error estimates}
               We begin with the introduction of a special Lagrange interpolation. For any 
               $v \in \mathcal{H}_h$, let $v_I\in\mathcal U_h$ be 
                the Lagrange interpolation of $v$ satisfying the  conditions  
               \begin{equation*}
                   v_I\left(x_{i, j}\right)=v\left(x_{i, j}\right), \quad 1\le i \le N, 1\le j\le k+1. 
               \end{equation*}
                The standard approximation theory gives us 
               \begin{equation*}
               \left\|v-v_I\right\|_{0} \lesssim h^{k+1}\|v\|_{k+1}.
               \end{equation*}
               The optimal error estimate for both OF-LSV and OF-RRSV methods is  presented below. 
               \begin{theorem}\label{optimalanalisis}
                   Assume $u(\cdot,t) \in H^{k+2}([a,b])$ is the exact solution of \eqref{1dlinear} and $u_h$  is the solution of \eqref{Gar}  with the initial solution  chosen as $ u_h(x, 0)=(u_0)_I(x)$.     Then 
                   \begin{equation}\label{opti-OFLSV}
                       \left\|(u-u_h)(\cdot, t)\right\|_0 \leq C h^{k+1},
                       \end{equation}
                       where C is a constant, depending on $u$ and its derives up to $(k+2)$-$th$ order.
               \end{theorem}
            \begin{proof}
                   Let 
               \[
                 e=u-u_h,\ \  \xi= u_I-u_h, \eta=u_I-u. 
               \]
                   Since the exact solution $u$  satisfies $a_h^{SV}(u,v^*)=0,\ v\in \mathcal U_h$, there holds  the following orthogonality    
                   \begin{equation*}
                       a_{h}^{SV}(e,v^*) = \sum_{i=1}^N\sum^k_{l=0} \frac{\sigma_{i}^{l}}{h_{i}} \left(u_h-P_{h}^{l-1} u_h, v^* \right)_{\tau_{i}},\ \ v\in \mathcal{U}_h. 
                   \end{equation*}
                   On the other hand,  we choose $v=\xi$ in \eqref{LSVstability} and use \eqref{LSVstability:0} and the  above orthogonality, and then obtain 
                   \begin{equation}\label{errorequation1}
                    (\xi_t,\xi^*) + \sum^N_{i=1} \sum^k_{l=0} \frac{\sigma_{i}^{l}}{h_{i}} \Vert \xi-P_{h}^{l-1} \xi \Vert_{0,\tau_i} ^2 \le   
                    a_{h}^{SV}(\eta,\xi^*) + \sum_{i=1}^N\sum^k_{l=0} \frac{\sigma_{i}^{l}}{h_{i}} \left(u_I-P_{h}^{l-1} u_I, \xi^* \right)_i. 
                \end{equation}
                We next estimate the two terms appeared in the right hand side of \eqref{errorequation1}.
                As for $a_{h}^{SV}(\eta,\xi^*)$,   there holds the conclusion in \cite{SV-Cao} 
                \begin{equation}\label{LSVoptimate}
                 | a_{h}^{SV}(\eta,\xi^*)  |\lesssim h^{k+1}\|\xi^*\|_0.
                \end{equation}
                 Using the  Cauchy-Schwarz inequality and the approximation property of $u_I$,  we get 
                 \[
                    (u_I-P_{h}^{l-1} u_I, \xi^* )_i\le \|u_I-P_{h}^{l-1} u_I\|_{0,{\tau}_i}\|\xi\|_{0,\tau_i}\lesssim h^{\max (1, l)+\frac{1}{2}}\|u_I\|_{l,\infty,\tau_i}\|\xi\|_{0,\tau_i}.
                 \]
                 Recalling the definition of $\sigma_i^l$ in \eqref{xigema_ij},  there holds 
                 \begin{equation}\label{xigemaest}
                   \begin{aligned}
                   \left(\sigma_i^l\right)^2 &=\frac{4(2 l+1)^2}{(2 k-1)^2} \frac{h^{2 l}}{(l !)^2}\left(\llbracket \partial_x^l\left(u_h-u\right) \rrbracket_{i-\frac{1}{2}}^2+\llbracket \partial_x^l\left(u_h-u\right) \rrbracket_{i+\frac{1}{2}}^2\right) \\
                   & \lesssim h^{2 l}\left(\llbracket \partial_x^l \xi \rrbracket_{i-\frac{1}{2}}^2+\llbracket \partial_x^l \xi \rrbracket_{i+\frac{1}{2}}^2\right)+h^{2 l}\left(\llbracket \partial_x^l \eta \rrbracket_{i-\frac{1}{2}}^2+\llbracket \partial_x^l \eta \rrbracket_{i+\frac{1}{2}}^2\right) .
                   \end{aligned}
               \end{equation}
                Combining the last two inequalities yields 
                
                \begin{equation}\label{errorh2}
                  \begin{aligned}
                     &\left|\sum_{i=1}^N\sum^k_{l=0} \frac{\sigma_{i}^{l}}{h_{i}} \left(u_I-P_{h}^{l-1} u_I \xi^* \right)_i \right|\\ 
                     \lesssim  &\sum_{i=1}^N \sum_{l=0}^k h_i^{\max \left(l-\frac{1}{2}, \frac{1}{2}\right)+l}\left(\llbracket \partial_x^l \xi \rrbracket_{i-\frac{1}{2}}^2+\llbracket \partial_x^l \eta \rrbracket_{i-\frac{1}{2}}^2\right)^{\frac{1}{2}}\|\xi^*\|_{0,\tau_i}\\
                     \lesssim  &\|\xi\|_0\|\xi^*\|_0+h^{k+1}\|\xi^*\|_0. 
                  \end{aligned}
              \end{equation}
              Substituting \eqref{LSVoptimate} and \eqref{errorh2} into \eqref{errorequation1} leads to 
              
              \begin{equation*}
                (\xi_t,\xi^*) \lesssim  \|\xi\|_0\|\xi^*\|_0+h^{k+1}\|\xi^*\|_0.
                 \end{equation*}
                 By using the $L^2$ equivalence in \eqref{equi:1} and the Gronwall's inequality, we have
             \begin{equation}\label{xi}
                 \|\xi(\cdot, t)\|_0 \lesssim\|\xi(\cdot, 0)\|_0+h^{k+1}.
             \end{equation}
             Then \eqref{opti-OFLSV} follows from the approximation property of $u_I$. 
            \end{proof}
             
             \subsection{Superconvergence}

             Following the argument in \cite{SV-Cao}, we adopt the idea of correction function to study the superconvergence 
             property of the OF-LSV and OF-RRSV.  
             
              We begin with the construction of the  correction function $\omega \in\mathcal U_h$, which  is defined in each 
              element $\tau_i, 1\le i\le N$  by 
             \begin{equation}\label{omega}
                 \left\{\begin{array}{l}
                 \left(\omega, v_x\right)_i = \left(\partial_t \eta, v^*\right)_i + \displaystyle\sum^k_{l=0} \frac{\sigma_{i}^{l}}{h_{i}} \left(u_I-P_{h}^{l-1} u_I, v^* \right)_i, \quad \forall v \in \mathcal{P}_{-}(\tau_i), \\
                 \omega(x_{i+\frac{1}{2}}^{-},t)=0,
                 \end{array}\right.
             \end{equation}
             where $\mathcal{P}_{-}$ is the orthogonal complement of $\mathcal{P}^0$ in $\mathcal{P}^k$, namely, $\mathcal{P}^k= \mathcal{P}^0\bigoplus\mathcal{P}_{-}$.
             \begin{lemma}
                 The correction function $\omega$ defined by \eqref{omega} is uniquely determined. Moreover, if $u \in W^{k+3,\infty}([a,b])$, then      \begin{equation}\label{omegaestimate}
                     \|\partial_t^r\omega\|_0\leq C h^{k+2}, \quad r=0,1.
                 \end{equation}
             where C is a constant, depending on the exact solution $u$ and its derivative up to $(k+2+r)$-$th$.
             \end{lemma}
             \begin{proof}
                 Since $\omega|_{\tau_i}\in\mathcal{P}^k$,  we suppose that
              \begin{equation*}
                     \omega|_{\tau_i}=\sum_{m=0}^k c_{i, m}(t) L_{i, m}(x),
                 \end{equation*}
             where $L_{i,m}$ denotes the Legendre polynomial of degree $m$ in  $\tau_i$.  
              Denoting $\phi_{i,m+1} = \frac {2} {h_i} \int_{x_{i-\frac 1 2}}^x L_{i,m} dx$
              and  choosing $v=\phi_{i,m+1}, m=0,\cdots, k-1 $ in \eqref{omega} yields 
                 \begin{equation*}
                  \frac{2}{h_i} \left(\omega, L_{i, m}^*\right)_i= (\partial_t \eta, \phi_{i, m+1}^*)_i+\displaystyle\sum^k_{l=0} \frac{\sigma_{i}^{l}}{h_{i}} \left(u_I-P_{h}^{l-1} u_I, \phi_{i, m+1}^* \right)_i:={\cal H}_1+{\cal H}_2. 
                 \end{equation*}
             %
             As for ${\cal H}_1$,  it was proved in  \cite{SV-Cao} that  
             \begin{equation}\label{H1}
                {\cal H}_1 
                \lesssim h^{m^{\prime}}\|u\|_{m^{\prime}, \infty, {\tau}_i},\quad m^{\prime} = \max  (2k+1-m, k+2).
             \end{equation}
               To estimate ${\cal H}_2$, we have, from \eqref{integelerror}, 
             \begin{equation}\label{H2}
                 {\cal H}_2=\displaystyle\sum^k_{l=0} \frac{\sigma_{i}^{l}}{h_{i}} ((u_I-P_{h}^{l-1} u_I, \phi_{i, m+1})_i + \frac{2}{h_i}R_i(L_{i,m}V_l)),\ {\rm with}\ V_l= D_x^{-1}(u_I-P_{h}^{l-1} u_I). 
             \end{equation}
             Using the fact that right Radau  quadrature is exact for $2k$, we have $R_i(L_{i,m}V_l)=0$ for all $0\le m\le k-1$ for OF-RRSV. 
             As for OF-LSV, we use the residual of Gauss-Legendre quadrature to derive 
             %
             \begin{equation*}
                 \begin{aligned}
                R_i(L_{i,m}V_l)=R_i(L_{i,m} D_x^{-1} u_I)=
                 (\frac{h_i}{2})^{2k+1} c_l^k\partial_x^{2k}(L_{i,m}D_x^{-1} u_I)(\xi), \xi\in \tau_i,
                 \end{aligned}
             \end{equation*}
                 which yields, together with the inverse inequality that 
              \begin{equation*}
                        \frac{2}{h_i} |R_i(L_{i,m}V_l)|  \lesssim (h_i)^{2k}\|\partial _x^{k-1}L_{i,m}\|_{0,\infty,{\tau}_i} \|\partial_x^ku_I\|_{0,\infty,\tau_i}\\
                          \lesssim h_i^{k+1} \|u\|_{k,\infty,\tau_i}. 
             \end{equation*}
              As a direct consequence of the Cauchy-Schwarz inequality and  the approximation property of $P_h^l$,  
             \begin{equation*}
                     \left\lvert \left(u_I-P_{h}^{l-1} u_I, \phi_{i, m+1} \right) _i\right\rvert 
                     \lesssim h_i^{\max\{1,l\}+1} \|u\|_{l,\infty,\tau_i}. 
             \end{equation*}
             In light of \eqref{xigemaest} and the estiamte of  $\xi$ in \eqref{xi},  we have 
             \begin{equation*}\label{xigemafinalest}
                 \left|\sigma_i^l\right|\lesssim  h^{k+1}\|u\|_{k+2,\infty, \tau_{i-1}\cup \tau_{i}\cup \tau_{i+1}}. 
             \end{equation*}
             Substituting the last three inequality into the formulation of ${\cal H}_2$ in \eqref{H2} yields 
             \begin{equation}\label{W2}
                 \left\lvert  {\cal H}_2 \right\rvert\lesssim \sum_{l=0}^k \left|\sigma_i^l\right|  \cdot h_i^{-1}\cdot ( h_i^{\max\{1,l\}+1} + h_i^{k+1}) \lesssim h_i^{k+2}\|u\|_{k+2,\infty, \tau_{i-1}\cup \tau_{i}\cup \tau_{i+1}}.
             \end{equation}
             Combining  the estimates of ${\cal H}_i, 1\le i\le 2$ in \eqref{H1} and \eqref{W2}, we have 
             $$
             \begin{aligned}
             \left\lvert c_{i,m}\right\rvert &\lesssim \frac{2}{h_i} |\left(\omega, L_{i, m}^*\right)_i | \lesssim  h_i^{k+2}\|u\|_{k+2,\infty,\tau_{i-1}\cup \tau_{i}\cup \tau_{i+1}},\ \ m\le k-1.
             \end{aligned}
             $$
             As for $m=k$, the identity $\omega(x_{i+\frac 1 2})^-=0$ implies that
             $$\left\lvert c_{i,m}\right\rvert  = \left\lvert \sum_{n=0}^{k-1}c_{i,n}\right\rvert  \lesssim h_i^{k+ 2}\|u\|_{k+2,\infty,\tau_{i-1}\cup \tau_{i}\cup \tau_{i+1}}.$$
             Consequently,
             $$
             \left\lVert \omega \right\rVert ^2_0 \lesssim \sum_{i=1}^N h_i \sum_{m=0}^k \left\lvert c_{i,m}\right\rvert ^2 \lesssim h^{2(k+2)}\|u\|_{k+2,\infty}.
             $$
             The \eqref{omegaestimate} is valid for $r=0$. 
             Taking time derivative in both sides of \eqref{omega} and following the same argument,  we can prove that  \eqref{omegaestimate}
             also holds true for $r=1$.  
       \end{proof}

             \begin{theorem}
                 Let $u\in W^{k+3,\infty}(\Omega)$ be the solution of \eqref{1dlinear}, $u_h$ be the solution of \eqref{Gar}  with the initial solution  chosen as $ u_h(x, 0)=\widetilde{(u_0)}_I(x)$.  Then  for both OF-LSV and OF-RRSV, 
              \begin{equation}\label{eq:2}
                \| (u_I - u_h)(\cdot, t)\|_0 \leq C h^{k+2},
                 \end{equation}
                 where C is a constant, depending on the exact solution $u$ and its $(k+3)$-$th$ derivative.
              \end{theorem}
             \begin{proof}      Denoting 
             \[
              \widetilde{u}_I = u_I-\omega,\quad e=u-u_h,\quad   \xi  = u_h-\widetilde{u}_I,\quad  \eta  = u-\widetilde{u}_I.
              \] 
               On the one hand,  following the same argument as that in \eqref{errorequation1}, we have 
                \begin{equation}\label{eq:1}
                     ( \xi_t, \xi^*) + \sum^N_{i=1} \sum^k_{l=0} \frac{\sigma_{i}^{l}}{h_{i}} \Vert  \xi-P_{h}^{l-1}  \xi \Vert_0 ^2  \le    
                     a_{h}^{SV}(\eta, \xi^*) + \sum_{i=1}^N\sum^k_{l=0} \frac{\sigma_{i}^{l}}{h_{i}} \left(\tilde u_I-P_{h}^{l-1} \tilde u_I,\xi^* \right)_i. 
                 \end{equation}
                On the other hand, we have for all $v\in \mathcal{P}_{-}$,   from \eqref{omega} that 
              \begin{eqnarray*}
                          a_{hi}^{SV}(\eta,v^*) + \sum^k_{l=0} \frac{\sigma_{i}^{l}}{h_{i}} \left(\widetilde{u}_I-P_{h}^{l-1} \widetilde{u}_I, v^* \right)_i
                        & = & (\partial_t\omega,v^*)_i+ \sum_{l=0}^k \frac { \sigma _i^l}{h_i}\left(\omega - P_h^{l-1}\omega,v^*  \right)_i \\
                         \leq  \left\lVert \partial_t \omega \right\rVert_{0,\tau_i} \left\lVert v^*\right\rVert _{0,\tau_i} &+& \sum_{l=0}^k \frac { \sigma _i^l}{h_i}\left\lVert \omega - P_h^{l-1}\omega\right\rVert_{0,\tau_i} \left\lVert v^*\right\rVert _{0,\tau_i}. 
                 \end{eqnarray*} 
                  Summing up all $i$ and using \eqref{omegaestimate}  gives 
             \[
                  a_{h}^{SV}(\eta,v^*) + \sum_{i=1}^N\sum^k_{l=0} \frac{\sigma_{i}^{l}}{h_{i}} \left(\widetilde{u}_I-P_{h}^{l-1} \widetilde{u}_I, v^* \right)_i
                  \lesssim  h^{k+2}\left\lVert v^*\right\rVert _{0},\ \ \forall v\in \mathcal{P}_{-}. 
             \]
             When $v\in \mathcal{P}^0$,  we easily obtain from the orthogonality of $P^l_h$ that 
         \[
            \sum^k_{l=0} \frac{\sigma_{i}^{l}}{h_{i}} \left(\widetilde{u}_I-P_{h}^{l-1} \widetilde{u}_I, v^*_0 \right)_i=\sum^k_{l=0} \frac{\sigma_{i}^{l}}{h_{i}} \left(\widetilde{u}_I-P_{h}^{l-1} \widetilde{u}_I, v_0 \right)_i=0.
         \]
         Moreover, since it was proved (see Theorem 5.4 in \cite{SV-Cao}) that 
     \[
       | a_{h}^{SV}(\eta,v^*_0) |\lesssim h^{k+2} \left\lVert v_0\right\rVert _{0},\ \forall  v_0\in \mathcal{P}^0, 
     \]
        then       
                \begin{equation*}
                   a_{h}^{SV}(\eta,v^*_0) + \sum_{i=1}^N\sum^k_{l=0} \frac{\sigma_{i}^{l}}{h_{i}} \left(\widetilde{u}_I-P_{h}^{l-1} \widetilde{u}_I, v^*_0 \right)_i \lesssim   h^{k+2}\left\lVert v_0^*\right\rVert _{0},\ \ v_0\in \mathcal{P}^0. 
                         \end{equation*} 
             Note that  all the function $v \in \mathcal{U}_h$ can be decomposed into $v = v_0+v_1$ with $v_0\in \mathcal{P}^0$ and $v_1\in \mathcal{P}_-$.  Consequently, 
             \begin{eqnarray*}
                          a_{h}^{SV}(\eta,v^*) + \sum_{i=1}^N\sum^k_{l=0} \frac{\sigma_{i}^{l}}{h_{i}} \left(\widetilde{u}_I-P_{h}^{l-1} \widetilde{u}_I, v^* \right)_i
                       \lesssim   h^{k+2}\left\lVert v^*\right\rVert _{0},\ \ \forall v\in\mathcal U_h. 
                 \end{eqnarray*}
               Substituting the above inequality into \eqref{eq:1},  we have 
             \begin{equation*}
                 \begin{aligned}
                      (\xi_t,\xi^*) \lesssim h^{k+2} \|\xi^*\|_0. 
                     \end{aligned}
             \end{equation*}
             Then \eqref{eq:2} follows from the 
             the Gronwall's inequality, the equivalence \eqref{equi:1} and  the estimate of $\omega$ in \eqref{omegaestimate}. 
            \end{proof}
               Thanks to the supercloseness result \eqref{eq:2} between $u_I$ and $u_h$, There holds  the following 
              superconvergence results for  
               the cell average error and error at downwind points. 
              
             \begin{theorem}\label{superconvergenceanalysis}
                Let $u\in H^{k+3}(\Omega)$ be the solution of \eqref{1dlinear}, $u_h$ be the solution of \eqref{Gar}  with the initial solution  chosen as $ u_h(x, 0)=\widetilde{(u_0)}_I(x)$.  Then  for both OF-LSV and OF-RRSV, 
                      \begin{equation}\label{superconvergenceresult}
                         \begin{aligned}
                         e_{n} &:=\left(\frac{1}{N} \sum_{i=1}^N\left(u-u_h\right)^2(x_{i+\frac{1}{2}}^{-},t)\right)^{\frac{1}{2}}\lesssim h^{k+2},\\
                         e_{c} &:=\left(\frac{1}{N} \sum_{i=1}^N\left(\frac{1}{h_i} \int_{\tau_i}\left(u-u_h\right) d x\right)^2\right)^{\frac{1}{2}} \lesssim h^{k+2}. 
                         \end{aligned}
                         \end{equation}
             \end{theorem}
             \begin{proof}
                 By using the inverse inequality, 
              \begin{eqnarray*}
                 e_n = \left(\frac{1}{N} \sum_{i=1}^N\left(u_I-u_h\right)^2(x_{i+\frac{1}{2}}^{-},t)\right)^{\frac{1}{2}} 
                 \lesssim \|u_I-u_h\|_0. 
              \end{eqnarray*}
               Then the first inequality of \eqref{superconvergenceresult} follows from \eqref{eq:2}.  
               Similarly, since (see, \cite{SV-Cao})
              \[
                 \left(\frac{1}{N} \sum_{i=1}^N\left(\frac{1}{h_i} \int_{\tau_i}\left(u-u_I\right) d x\right)^2\right)^{\frac{1}{2}}\lesssim h^{k+2}, 
              \]
              then
             \[
               e_c = \left(\frac{1}{N} \sum_{i=1}^N\left(\frac{1}{h_i} \int_{\tau_i}\left(u-u_I+u_I-u_h\right) d x\right)^2\right)^{\frac{1}{2}}
               \lesssim  h^{k+2}+\|u_I-u_h\|_0, 
             \]
              which yields (together with \eqref{eq:2}) the second inequality of \eqref{superconvergenceresult}. 
            \end{proof}

             \section{Numerical experiments}

             In this section, the numerical tests will be presented to validate the accuracy, robustness of current scheme.
             The accuracy tests are provided for the linear advection equation and Euler solutions.
             The ninth-order strong stability preserving (SSP) Runge-Kutta method is applied as temporal discretization aimed at avoiding the interference of temporal
             discretization on the convergence rates.  For the flows with discontinuities,  a few benchmark cases for Euler solutions are provided.
             The specific heat ratio takes $\gamma = 1.4$, and the classic fourth order Runge-Kutta method is used for  temporal discretization.
             The CFL condition is
             \begin{equation*}
             \Delta t=\frac{C F L}{(\alpha +a_{0} )}\Delta x, ~ \alpha=\max _{i, s}\left|\left(\lambda_{s}\right)_{i+1/2}^{\pm}\right|, ~ a_{0}={\displaystyle \max _{i\in\mathcal{Z}_N}} \sum_{l=0}^{k} \sigma_{i}^{l},
             \end{equation*}
             where $\Delta t$ and $\Delta x$ are time step and cell size, $\left(\lambda_{s}\right)_{i+1/2}^{\pm}$
             are real eigenvalues of the Jacobian matrix $\partial\boldsymbol{F}/\partial\boldsymbol{U}$ at $x_{i+1/2}^{\pm}$ and $\sigma_{i}^{l}$ is the damping coefficient defined in \eqref{xigema_ij}.
             This reveals that the time step is severely limited by the coefficient $a_{0}$. In the 
             computation, the CFL number takes $2.4$.  
              $\mathcal{Q}_k$ with $k=2$ and $3$ are used for the accuracy tests, and 
            $\mathcal{Q}_k$  with $k=2$ is only used  for the cases with discontinuities.           
            The  mesh with uniform SVs are used for 1D cases, and
            the uniform rectangular meshes are used  for 2D cases.  Since the numberical solution of OF-LSV and OF-RRSV are resemblance, only the results of OF-LSV are presented. 
             
             \subsection{Accuracy tests}
In order to verify the order of optimal convergence and
superconvergence of current scheme, the one-dimensional and
two-dimensional cases are provided. In our numerical experiments,
the $L^2$ error $\|e_0\|_0$, the cell average error $e_c$ and the
average error at downwind points $e_n$ will be tested.

The first case is given for one-dimensional linear scalar equation
\eqref{1dlinear}, and the initial condition is set as follows
\begin{equation*}
u(x)=1+0.2 \sin (\pi x).
\end{equation*}
The computation domain $\Omega=[0,2]$ and the periodic boundary
condition is adopted at both ends. The analytic solution is
\begin{equation*}
u(x, t)=1+0.2 \sin (\pi(x-t)).
\end{equation*}
The uniform mesh with $N$ SVs are used. The errors and the order of convergence for
$\|e_0\|_0$, $e_c$ and $e_n$  are presented in Table
\ref{tab-1d-OFSV}  at $t=2$. The obtained order in Table
\ref{tab-1d-OFSV} are consistent with the analysis in Theorem \ref{optimalanalisis} and Theorem
\ref{superconvergenceanalysis},
i.e.,  $(k+1)$-$th$ optimal convergence order for $\|e_0\|_0$, both
$(k+2)$-$th$ convergence order for both $e_c$ and $e_n$. In order to
investigate the effect of damping term on accuracy, this example is
also tested by  the classical SV method with $k=2$ and $3$. The
errors and the order of convergence are listed in Table
\ref{tab-1d-SV}. The expected orders for SV method in \cite{SV-Cao}
are verified, i.e., the optimal convergence rate $(k+1)$ for the
$L^2$ error $\|e\|_0$, and the convergence rate $(2k)$ for both the
cell average error $e_c$ and the error at downwind point $e_n$. With
the mesh refinement, the $L^2$ error $\|e_0\|_0$ of OFSV method and
SV method get closer, which means the damping term becomes smaller
for the smooth analytic solution. However, the damping term does
affect on the superconvergence accuracy.

\begin{table}[htbp]
  \footnotesize
  \caption{One-dimensional accuracy
test: errors and convergence orders of $\|e_0\|_0$, $e_c$, $e_n$ for
OFSV scheme.}\label{tab-1d-OFSV}
  \begin{center}
    \begin{tabular}{c|c|cc|cc|cc} \hline
      $k$ &Mesh  &$\|e_0\|_0$  &order &$e_c$   &order        &$e_n$  &order   \\
      \hline
      ~ &$8$ &4.9119E-03 &~   &2.2847E-03 &~ &2.4315E-03 &~\\
      ~ &$16$ &3.5267E-04 &3.8000 &1.5955E-04 &3.8399 &1.6600E-04 &3.8727 \\
      2 &$32$ &2.4316E-05 &3.8583 &9.2332E-06 &4.1110 &9.3982E-06 &4.1426 \\
      ~ &$64$ &2.1921E-06 &3.4715 &5.3050E-07 &4.1214 &5.3672E-07 &4.1301 \\
      ~ &$128$ &2.4509E-07 &3.1609 &3.1495E-08 &4.0742 &3.1815E-08 &4.0764 \\
      \hline
      ~ &$8$ &2.0132E-04 &~   &1.3213E-04 &~ &1.3945E-04 &~ \\
      ~ &$16$ &6.2608E-06 &5.0070 &4.0326E-06 &5.0341 &4.1981E-06 &5.0539 \\
      3 &$32$ &2.3475E-07 &4.7371 &1.3517E-07 &4.8988 &1.3820E-07 &4.9249 \\
      ~ &$64$ &1.0447E-08 &4.4900 &4.5274E-09 &4.9000 &4.5788E-09 &4.9156 \\
      ~ &$128$ &5.5435E-10 &4.2361 &1.4816E-10 &4.9334 &1.4915E-10 &4.9401 \\ \hline
    \end{tabular}
  \end{center}
  \vspace{0.4cm}
  \caption{ One-dimensional accuracy test:  errors and convergence orders of $\|e_0\|_0$, $e_c$, $e_n$ for SV scheme.}\label{tab-1d-SV}
  \begin{center}
    \begin{tabular}{c|c|cc|cc|cc} \hline
      $k$ & Mesh  &$\|e_0\|_0$  &order &$e_c$   &order        &$e_n$  &order   \\
      \hline
      ~ &$8$ &1.1168E-03 &~ &2.8677E-04  &~ &3.0598E-04&~\\
      ~ &$16$ &1.2537E-04 &3.1551 &1.9091E-05 &3.9089 &1.9951E-05 &3.9389\\
      2 &$32$ &1.5151E-05 &3.0488 &1.2196E-06 &3.9647 &1.2674E-06 &3.9765\\
      ~ &$64$ &1.8767E-06 &3.0130 &7.6902E-08 &3.9872 &7.9807E-08 &3.9892\\
      ~ &$128$  &2.3405E-07 &3.0033 &4.8272E-09 &3.9938&5.0093E-09 &3.9938\\
      \hline
      ~ &$8$ &5.0192E-05 &~ &1.2012E-06 &~ &5.4416E-06 &~\\
      ~ &$16$  &2.9697E-06 &4.0791 &2.7583E-08 &5.4446 &3.0592E-08 &7.4747\\
      3 &$32$  &1.8558E-07 &4.0002 &4.3641E-10 &5.9820 &4.7041E-10 &6.0231\\
      ~ &$64$  &1.1599E-08 &4.0000 &6.8475E-12 &5.9939 &7.3725E-12 &5.9956\\
      ~ &$128$ &7.2493E-10 &4.0000 &2.9283E-13 &4.5474 &2.9877E-13 &4.6351\\\hline
    \end{tabular}
  \end{center}
  \end{table}

\begin{table}[htbp]
  \footnotesize
  \caption {Two-dimensional accuracy
  test: errors and convergence orders of $\|e_0\|_0$, $e_c$, $e_n$ for
  OFSV scheme.}\label{tab-2d-OFSV}
  \begin{center}
    \begin{tabular}{c|c|cc|cc|cc} \hline
      $k$ &Mesh  &$\|e_0\|_0$  &order &$e_c$   &order      &$e_n$  &order   \\
      \hline ~ &$8^2$ &1.8533E-02 &~  &9.7494E-03 &~&9.6293E-03 &~\\
      ~ &$16^2$ &1.2227E-03 &3.9219 &5.8527E-04 &3.9020 &6.2540E-04 &3.9446 \\
      2 &$32^2$ &7.4101E-05 &4.0445 &3.5010E-05 &4.0633 &3.6142E-05 &4.1130 \\
      ~ &$64^2$ &4.9616E-06 &3.9006 &2.0624E-06 &4.0854 &2.0998E-06 &4.1054 \\
      ~ &$128^2$ &4.2071E-07 &3.5617 &1.2461E-07 &4.0488 &1.2623E-07 &4.0560 \\
      \hline
      ~ &$8^2$ &1.4675E-03 &~  &6.7864E-04 &~&7.5641E-04 &~ \\
      ~ &$16^2$  &3.4150E-05 &5.4254 &1.6197E-05 &5.3889&1.7362E-05 &5.4452 \\
      3 & $32^2$ &1.0958E-06 &4.9619 &5.2147E-07 &4.9570 &5.4169E-07 &5.0023 \\
      ~ &$64^2$ &3.9061E-08 &4.8100 &1.7557E-08 &4.8925 &1.7918E-08 &4.9180 \\
      ~ &$128^2$ &1.5517E-09 &4.6539 &5.7939E-10 &4.9214 &5.5602E-10 &4.9343 \\ \hline
    \end{tabular}
  \end{center}
  \vspace{0.4cm}
  \caption{  Two-dimensional accuracy
  test:  errors and convergence orders of $\|e_0\|_0$, $e_c$, $e_n$
  for SV scheme.}\label{tab-2d-SV}
  \begin{center}
    \begin{tabular}{c|c|cc|cc|cc} \hline
      $k$ & Mesh &$\|e_0\|_0$  &order   &$e_c$   &order     &$e_n$  &order   \\
      \hline ~ & $8^2$ &1.7926E-03 &~&5.5948E-04  &~&6.1273E-04&~\\
      ~ & $16^2$ &1.8552E-04 &3.2724&3.8014E-05 &3.8795  &4.0014E-05 &3.9367\\
      2 & $32^2$ &2.1701E-05 &3.0957&2.4428E-06 &3.9599  &2.5469E-06 &3.9737\\
      ~ & $64^2$ &2.6629E-06 &3.0267&1.5457E-07 &3.9822  &1.6100E-07 &3.9836\\
      ~ & $128^2$ &3.3127E-07 &3.0069&9.7551E-09 &3.9859  &1.0190E-08 &3.9818\\
       \hline
      ~ & $8^2$ &7.1124E-05 &~&2.3412E-06 &~&1.0883E-06 &~\\
      ~ & $16^2$ &4.2004E-06 &4.0817 &5.4813E-08 &5.4166  &6.1185E-08 &7.4747\\
      3 & $32^2$ &2.6246E-07 &4.0004 &8.7141E-10 &5.9750  &9.4081E-10 &6.0231\\
      ~ & $64^2$ &1.6403E-08 &4.0000 &1.3685E-11 &5.9927  &1.4738E-11 &5.9963\\
      ~ & $128^2$ &1.0252E-09 &4.0000 &3.5470E-13 &5.2698  &3.6942E-13 &5.3182\\\hline
    \end{tabular}
  \end{center}
  \end{table}

The second  case for accuracy is the advection of density perturbation
for two-dimensional Euler equations, and the initial conditions are
given follows
\begin{equation*}
\rho_0(x,y)=1+0.2 \sin (\pi(x+y)),~ p_0(x,y)=1,~ U_0(x,y)=1,~ V_0(x,y)=1.
\end{equation*}
The computational domain is $[0,2]\times[0,2]$ and the periodic
boundary conditions are adopted in both directions. The analytic
solutions are
\begin{equation*}
\rho(x,y, t)=1+0.2 \sin (\pi((x+y)-2t)),~p(x,y, t)=1,~U(x,y, t)=1,~V(x,y, t)=1.
\end{equation*}
The uniform mesh with $N^2$ SVs are used in the computation. The errors and
convergence orders, including $\|e_0\|$,$e_c$,$e_n$, are presented
in Table.\ref{tab-2d-OFSV} at $t = 2$. The expected order of
accuracy are obtained.  As reference,  this case is also tested by
the classical SV method. The errors and
convergence orders are presented in Table.\ref{tab-2d-SV}. The
numerical results indicate the damping term would not pollute the
optimal order of accuracy and reduce the order of superconvergence for the Euler equations as well.

 \begin{figure}[htbp]
\centering
\includegraphics[width=0.45\textwidth]{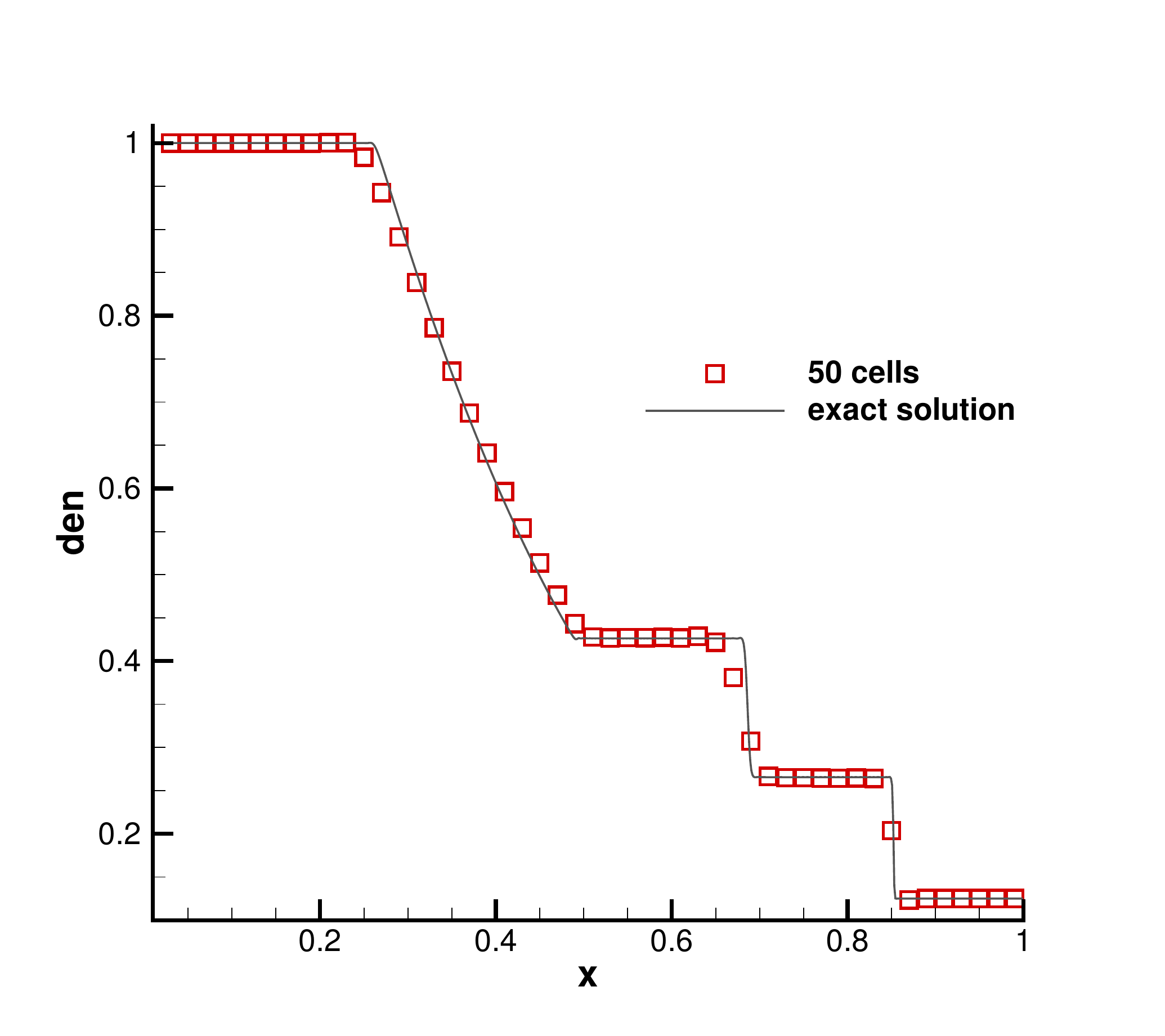}
\includegraphics[width=0.45\textwidth]{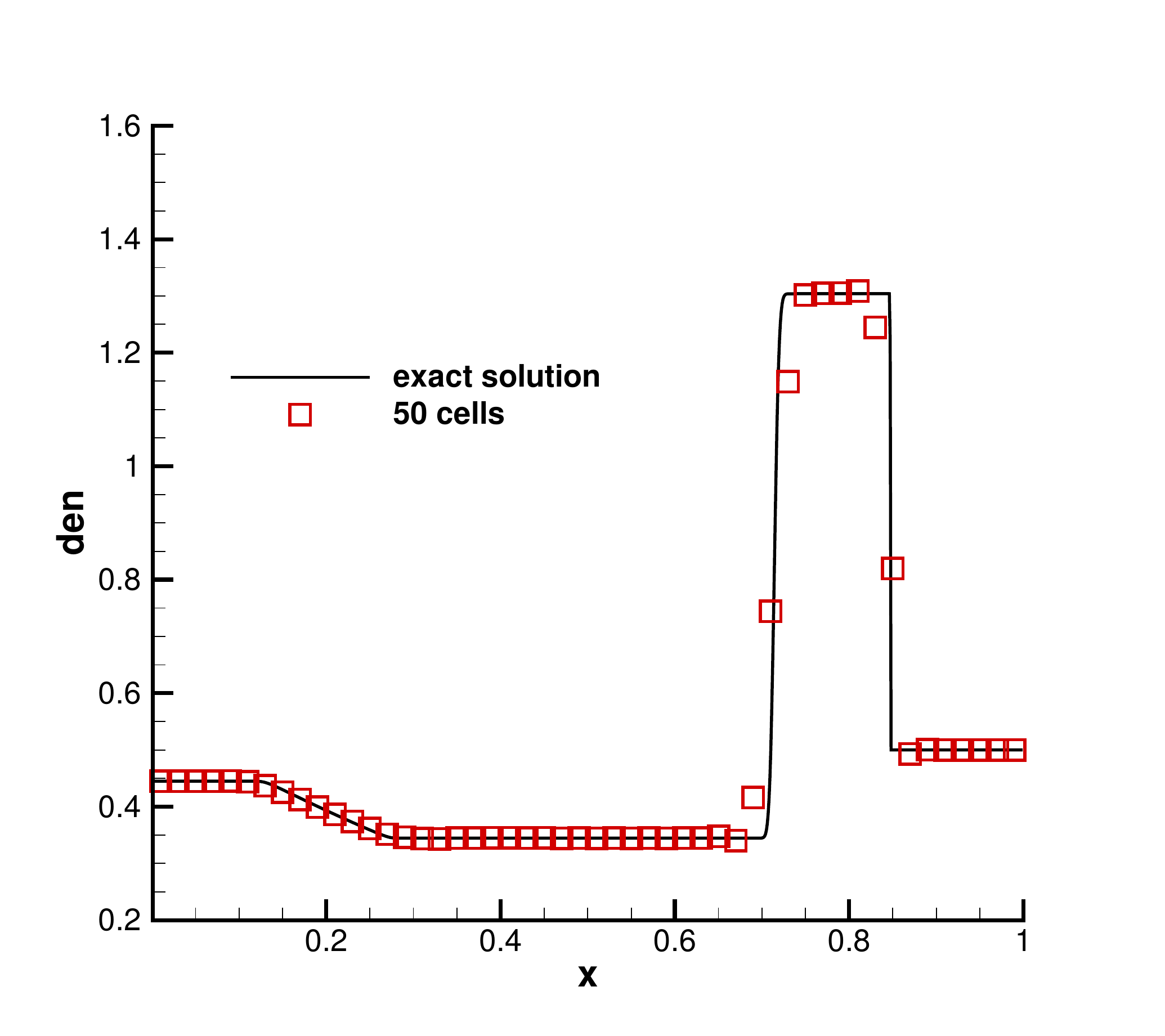}
\includegraphics[width=0.45\textwidth]{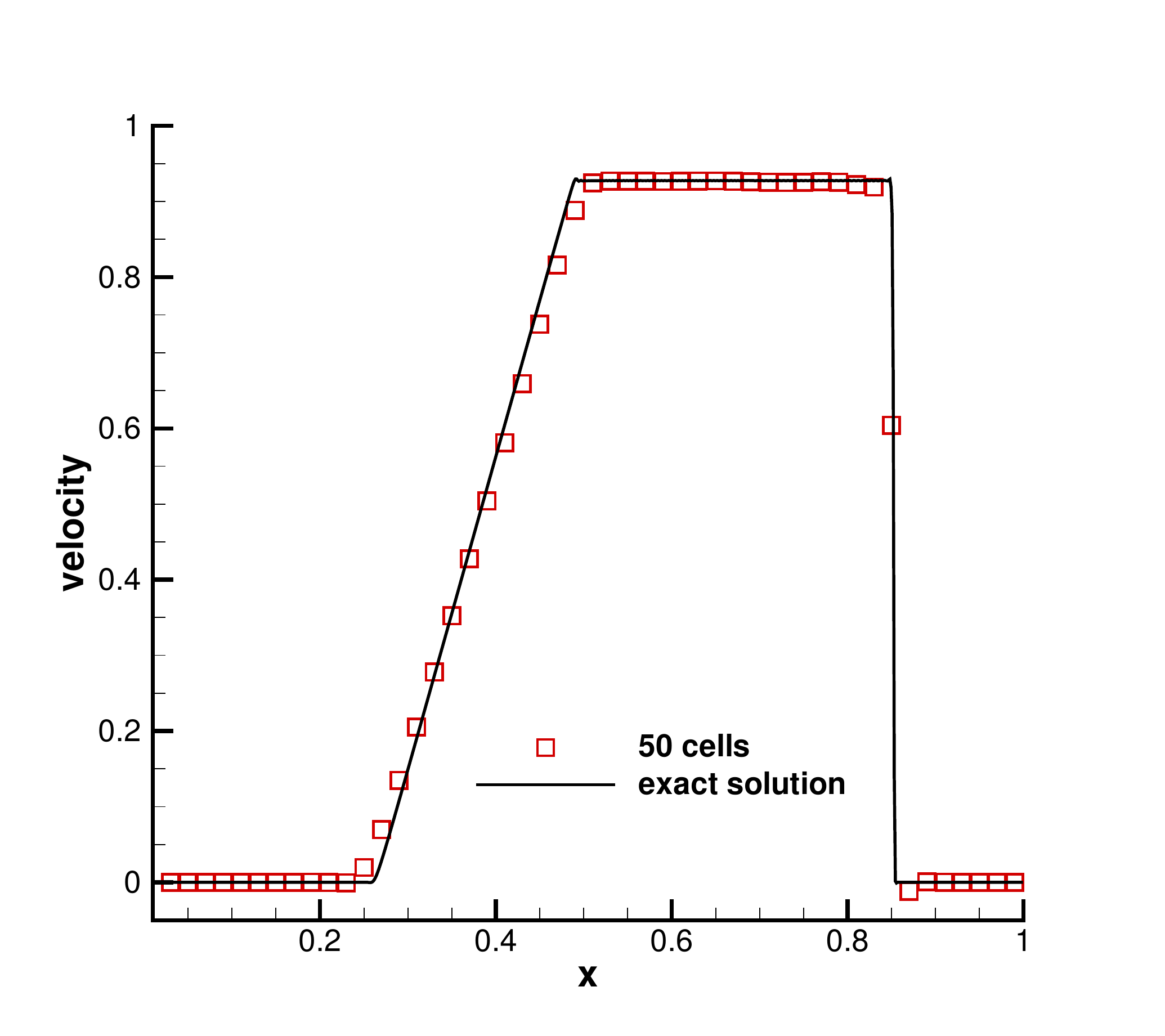}
\includegraphics[width=0.45\textwidth]{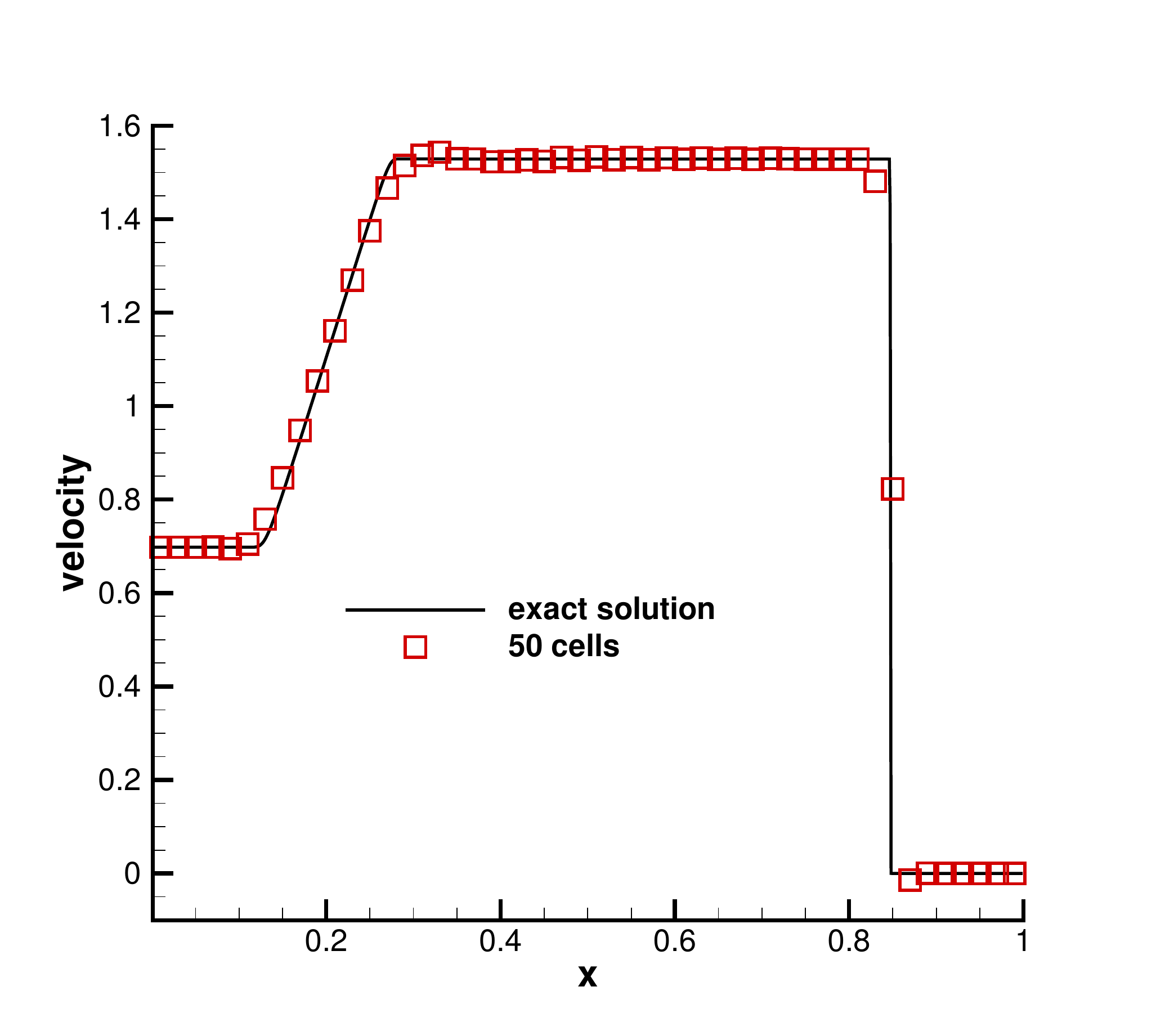}
\includegraphics[width=0.45\textwidth]{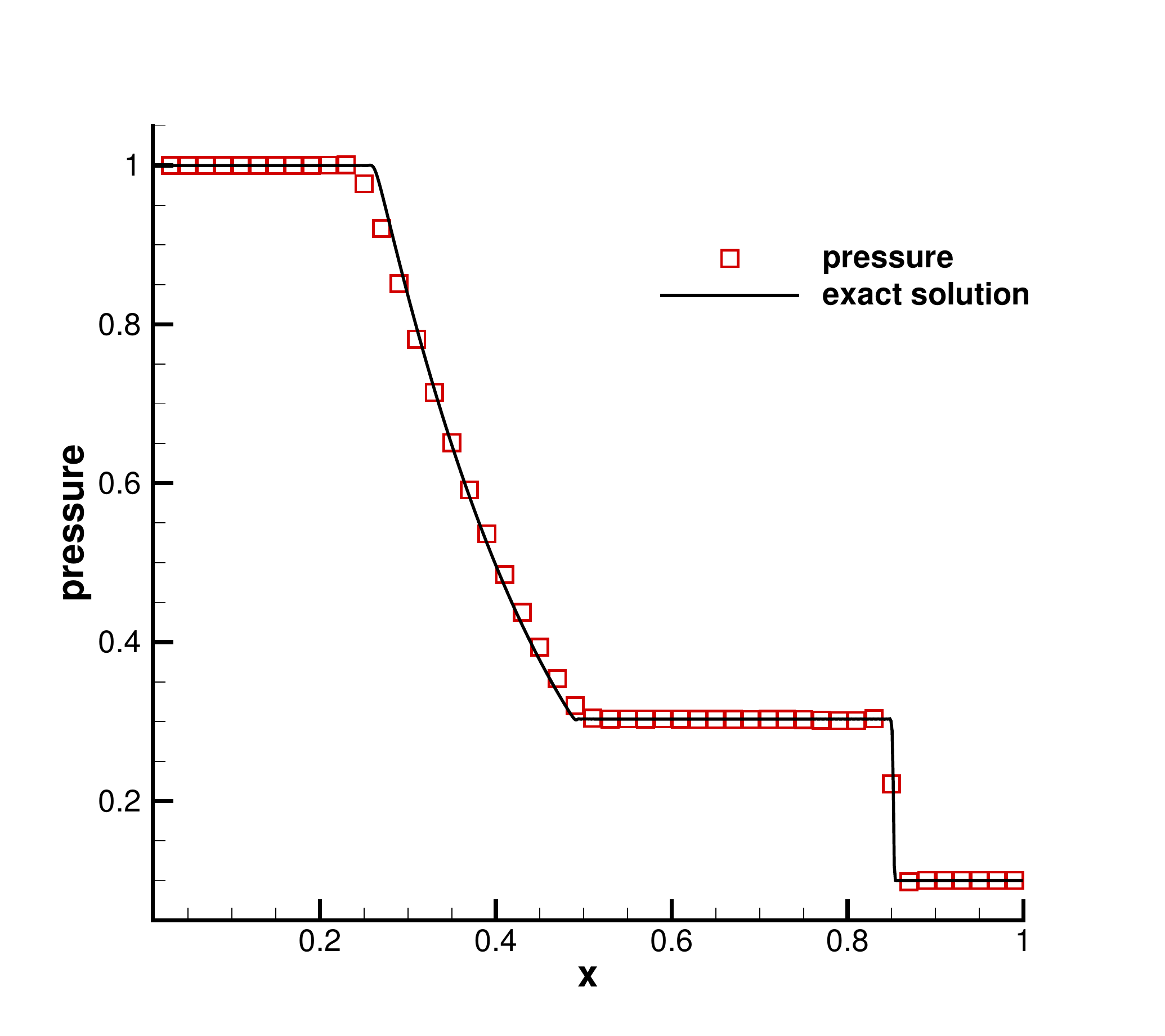}
\includegraphics[width=0.45\textwidth]{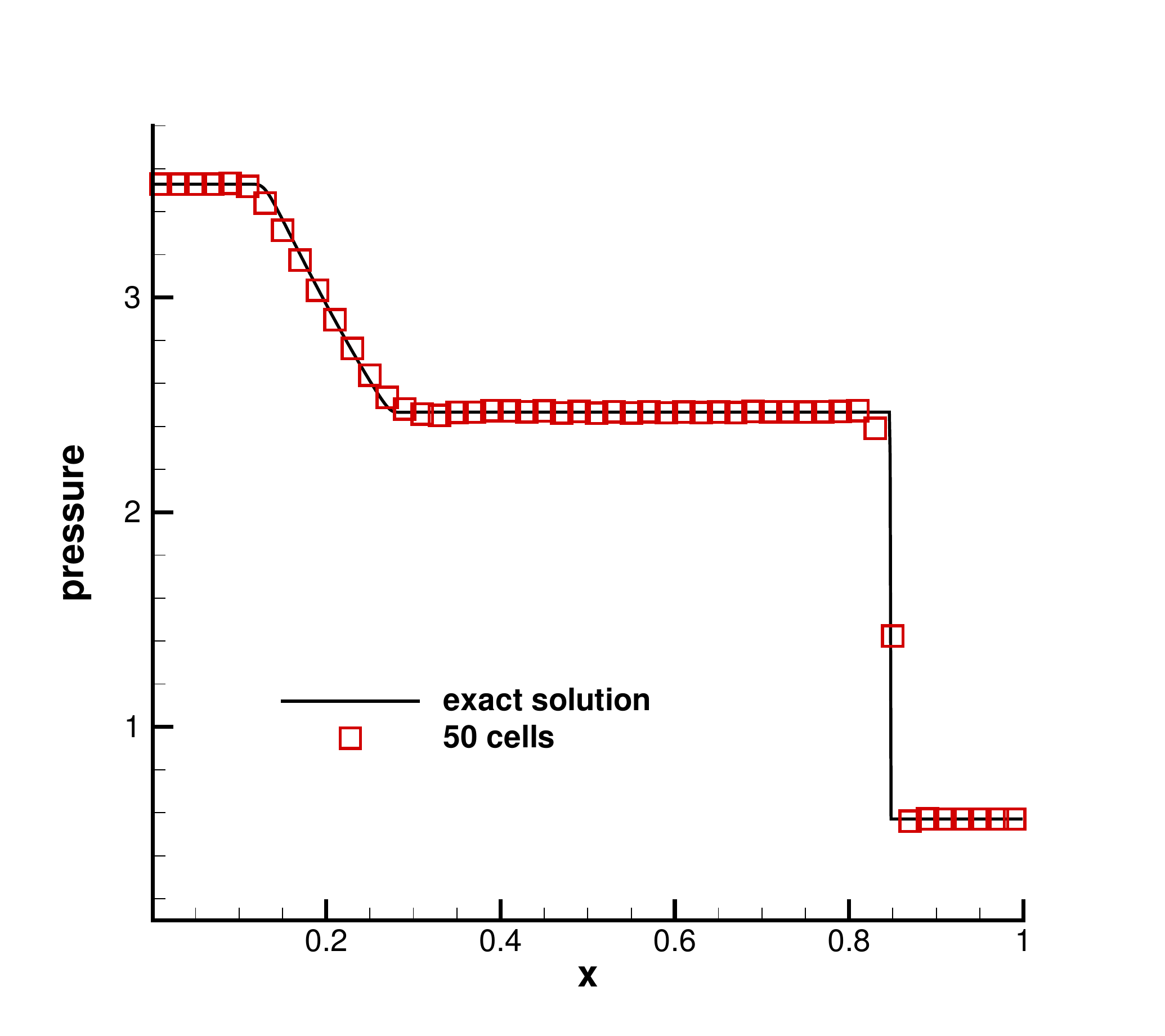}
\caption{\label{1d-riemann}  One-dimensional Riemann problem:
the density, velocity and pressure distributions for Sod problem (left) at $t = 0.2$, and  Lax problem (right)  at $t = 0.14$.}
\end{figure}

\subsection{One-dimensional Riemann problems}
In this case, two Riemann problems of one-dimensional Euler equations are tested.
The first test is Sod problem and the initial conditions are given as follows
 \begin{equation*}
(\rho, U, p)=\begin{cases}
(1,0,1),  &0 \leq x<0.5,\\
(0.125,0,0.1), &0.5 \leq x \leq 1.
\end{cases}
\end{equation*}
The second one is Lax problem and the initial conditions are given as follows
\begin{equation*}
(\rho, U, p)=\begin{cases}
(0.445,0.698,3.528), & 0 \leq x<0.5,\\
(0.5,0,0.571), & 0.5 \leq x \leq 1.
\end{cases}
\end{equation*}
For these two tests, the computational domain is $[0,1]$ with 50
uniform SVs and non-reflecting boundary condition is adopted on
both ends. The cell average of density, velocity, and pressure
distributions for the third-order OFSV method and the exact
solutions are presented in Figure.\ref{1d-riemann}  for Sod problem
at $t = 0.2$ and for Lax problem at $t = 0.14$. The numerical
results agree well with the exact solutions and the spurious
oscillations are effectively restrained.

\begin{figure}[!h]
\centering
\includegraphics[width=0.475\textwidth]{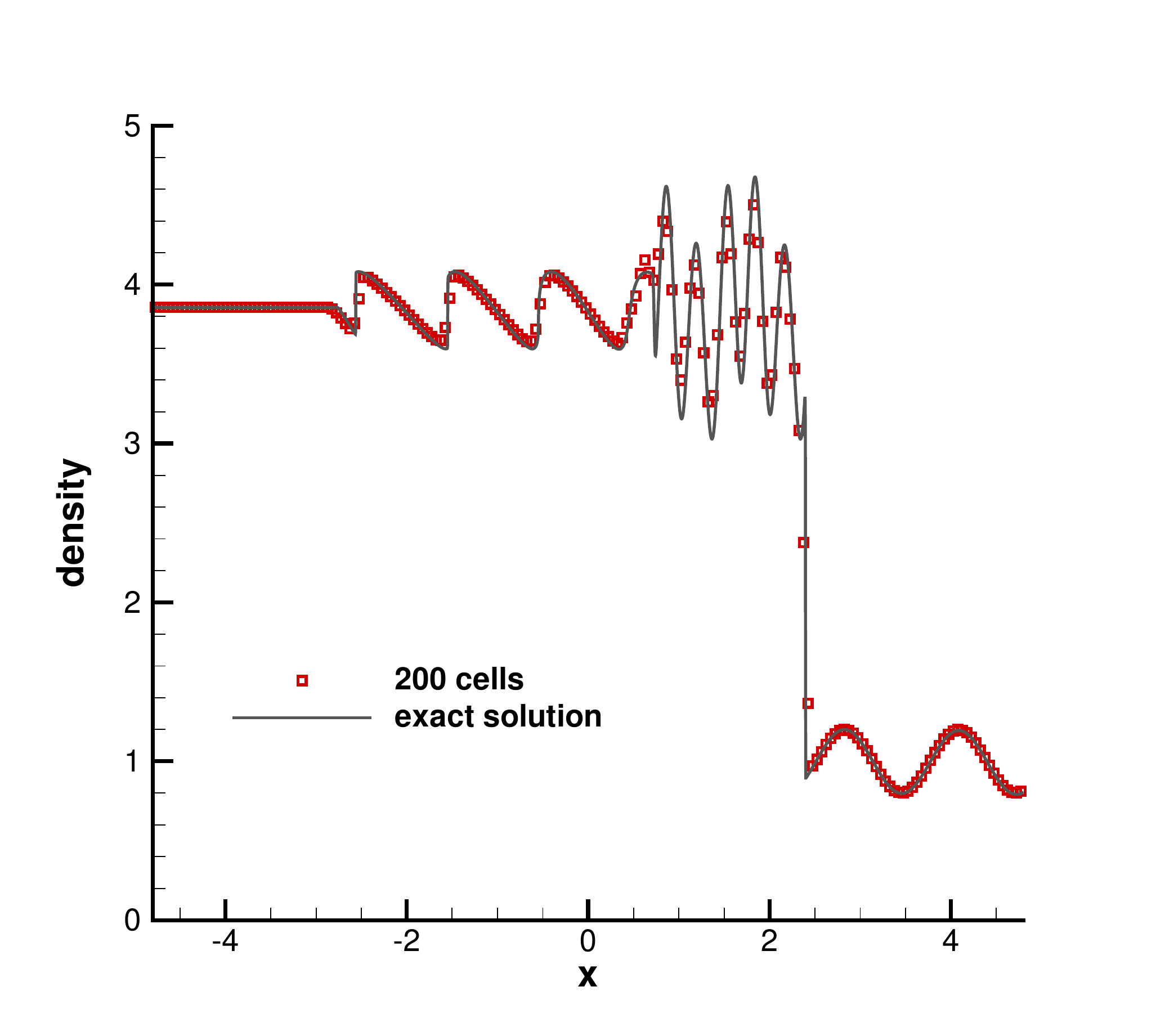}
\includegraphics[width=0.475\textwidth]{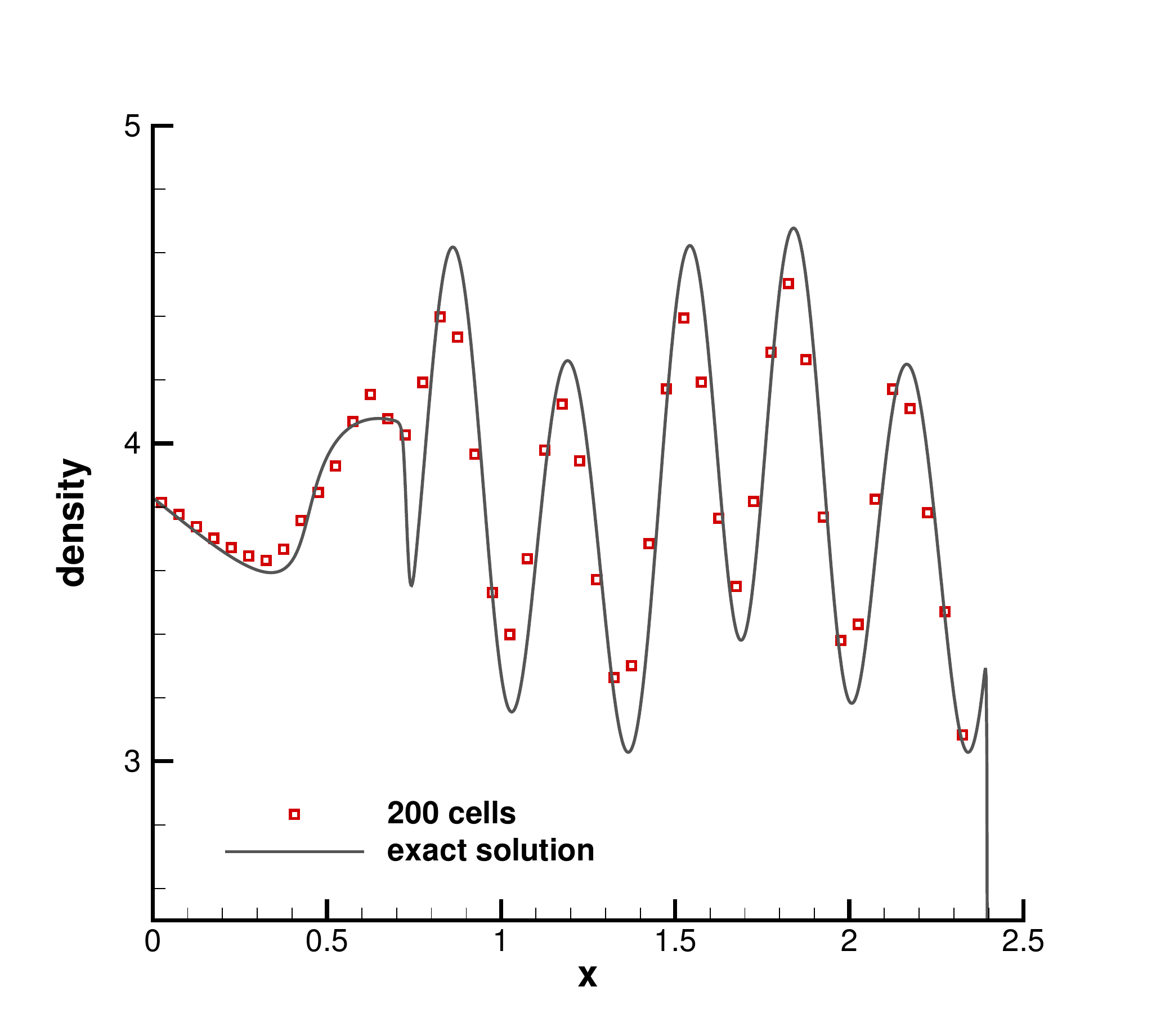}
\caption{\label{Shu-Osher} Shu-Osher problem: the density distribution and local enlargement at $t = 2$.}
\centering
\includegraphics[width=0.475\textwidth]{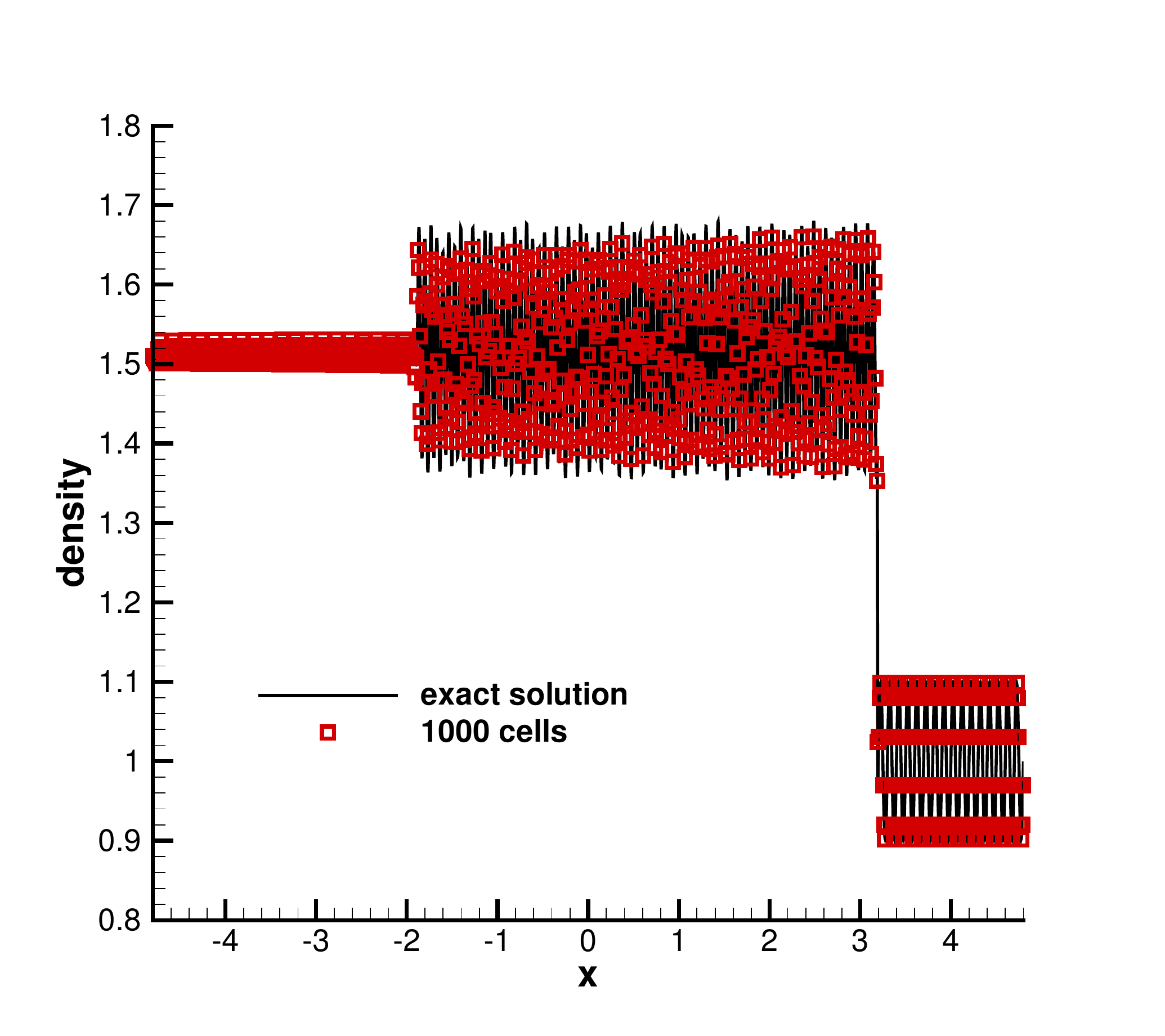}
\includegraphics[width=0.475\textwidth]{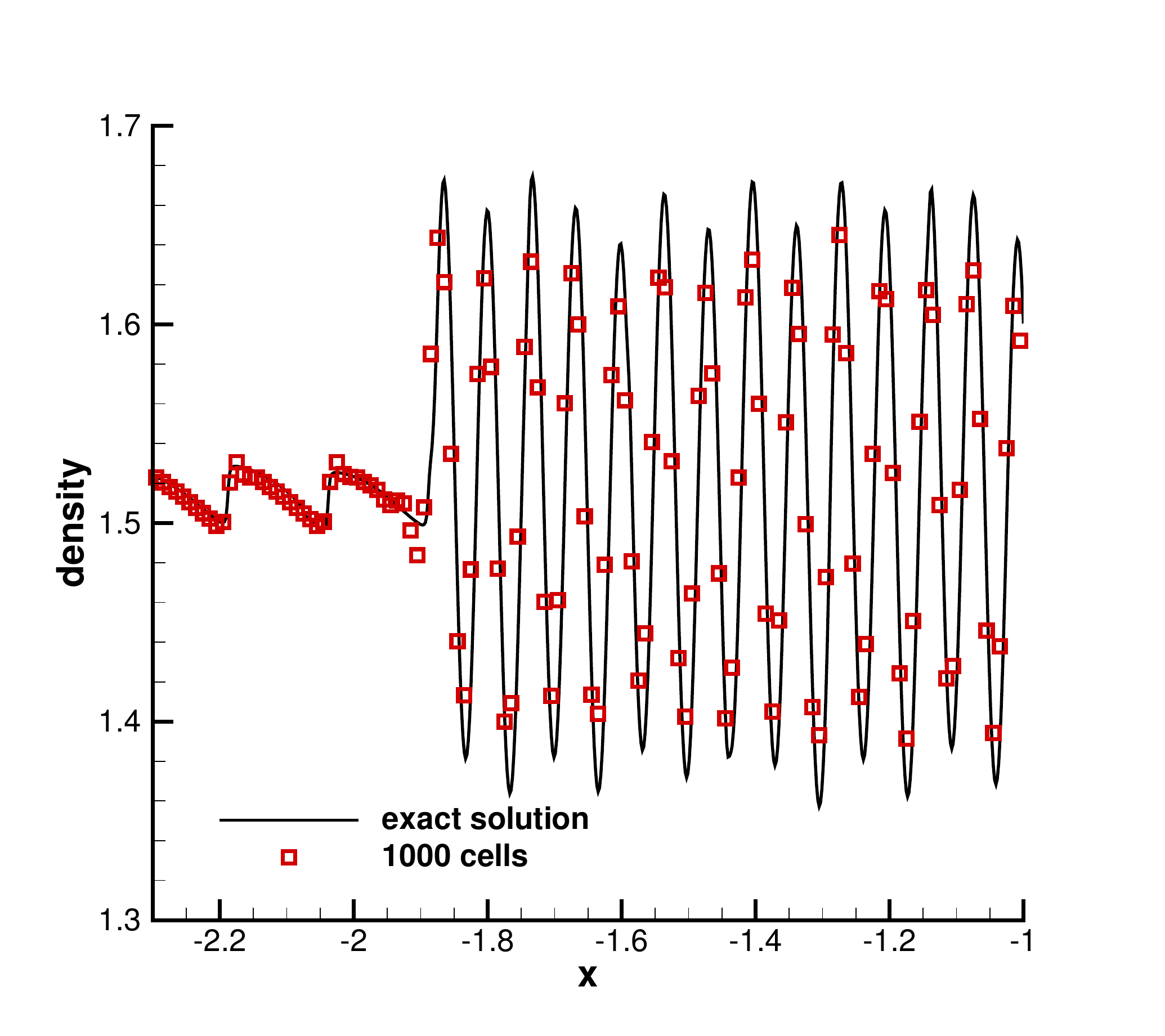}
\caption{\label{Titarev-Toro} Titarev-Toro problem: the  density distribution and local enlargement at $t = 5$.}
\end{figure}

\subsection{Shock-acoustic interactions}
For the one-dimensional case, another test case is the
Shu-Osher shock acoustic interaction \cite{ENO2}, describing the
interaction between a right moving  Mach $3$ shock interacting with a
sine wave in density. This is a typical example to show the
advantage of a high order scheme because both shocks and complex
smooth region structures coexist. The computational domain is taken
to be $[-5,5]$, and the initial conditions are given as
\begin{equation*}
(\rho, U, p)= \begin{cases}
(3.857134,2.629369,10.33333), & -5 \leq x \leq -4, \\
(1+0.2 \sin (5 x), 0,1), & -4<x<5.
\end{cases}
\end{equation*}
Uniform mesh with $200$ SVs is used. The cell average of density
distribution and local enlargement are presented in Figure.
\ref{Shu-Osher} at $t=2$, where the reference solution is given by
the fifth order finite volume WENO method with $1000$ cells. The
numerical solutions also agree well with the reference solution, which shows that the
current scheme controls the spurious oscillation effectively.

As an extension of Shu-Osher problem, Titarev-Toro shock  wave
interaction problem \cite{Case-Titarev}  is also tested, which
simulates a severely oscillatory wave interacting with shock. The
computational domain is also taken to be $[-5,5]$, and the initial
conditions for this case are given as follows
\begin{equation*}
(\rho, U, p)= \begin{cases}
(1.515695,0.523346,1.805), & -5<x \leq-4.5,\\
(1+0.1 \sin (20 \pi x), 0,1). & -4.5<x<5.
\end{cases}
\end{equation*}
Uniform mesh with $1000$ SVs is used. The density distribution  and
local enlargement at $t=5$ are presented  in Figure.\ref{Titarev-Toro},
where the reference solution is given by
the fifth-order finite volume WENO method with $4000$ cells.
The numerical result shows that  OFSV
method does have the ability of high-order numerical scheme to
capture the extremely high frequency waves and eliminate the
spurious oscillations.

\subsection{Two-dimensional Riemann problems}
In this case, two examples of two-dimensional Riemann problems \cite{Case-lax} are
presented, including the interactions of shocks and the interaction
of contacts with rarefaction waves. The computational domain is both
$[0,1] \times[0,1]$, and the non-reflecting boundary conditions are
used in all ends. The initial conditions for the first case are
\begin{align*}
(\rho, U, V, p)= \begin{cases}
(1.5,0,0,1.5), & x>0.5, y>0.5, \\
(0.5323,1.206,0,0.3), & x<0.5, y>0.5, \\
(0.138,1.206,1.206,0.029), & x<0.5, y<0.5, \\
(0.5323,0,1.206,0.3), & x>0.5, y<0.5.
\end{cases}
\end{align*}
A complicated pattern is evolved by the interaction between the four
initial shock waves. The density distributions and the local
enlargements are showed at $t=0.3$ in Figure.\ref{riemann-1} with
$400 \times 400$ and $800 \times 800$ uniform SVs. The numerical
results show that the small scale flow structures are well captured
by the current scheme.

\begin{figure}[!h]
\centering
\includegraphics[width=0.475\textwidth]{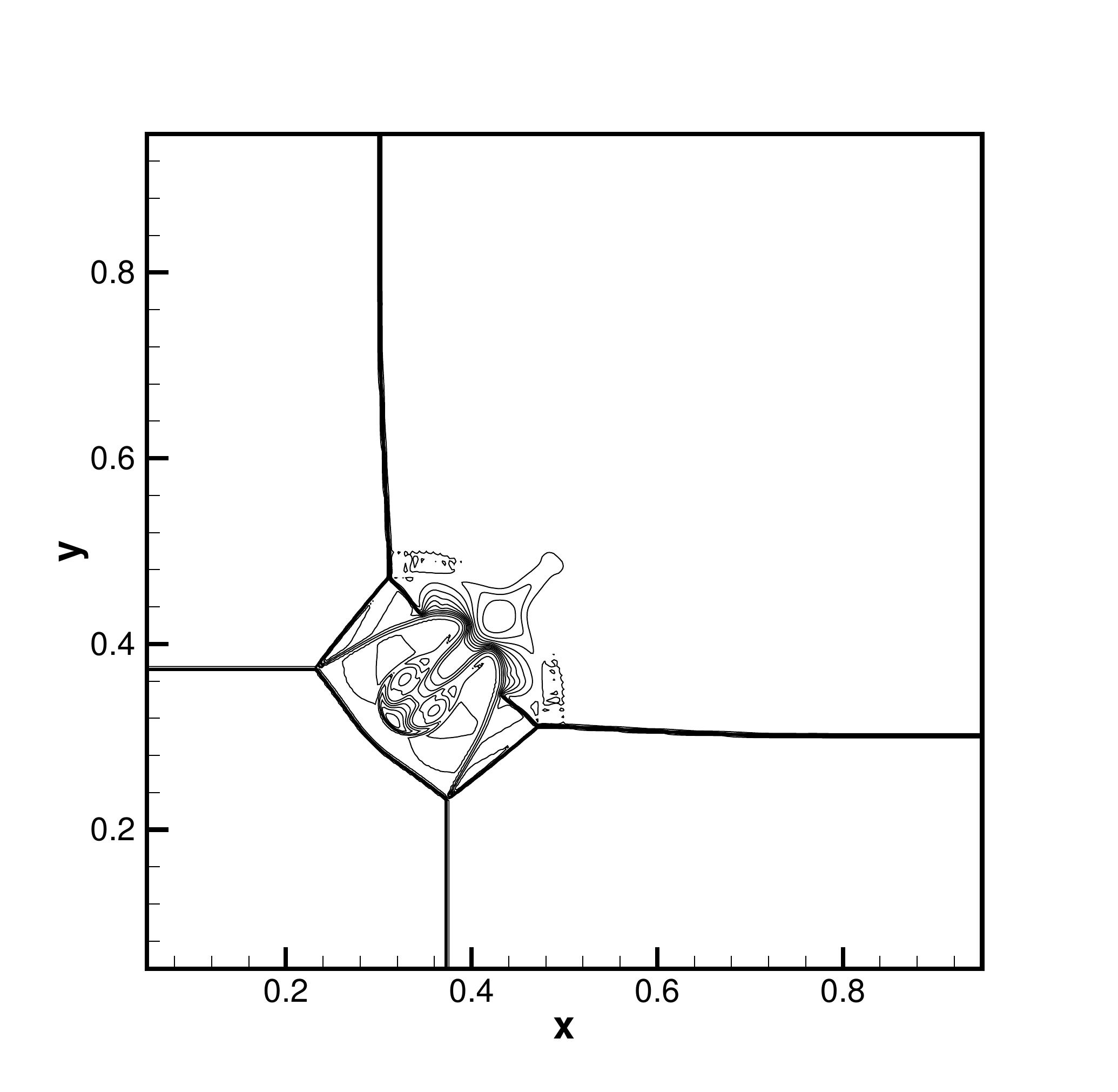}
\includegraphics[width=0.475\textwidth]{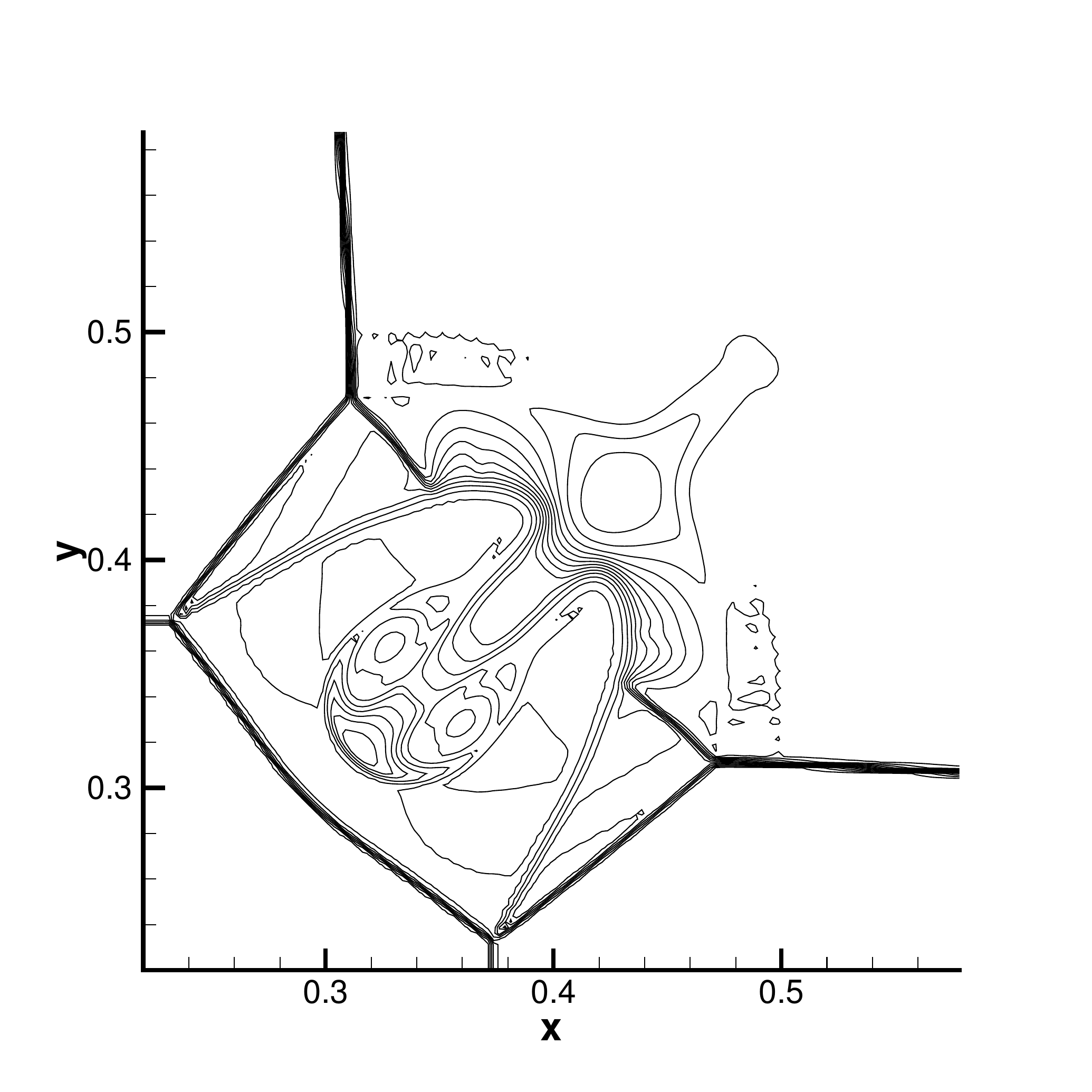}
\includegraphics[width=0.475\textwidth]{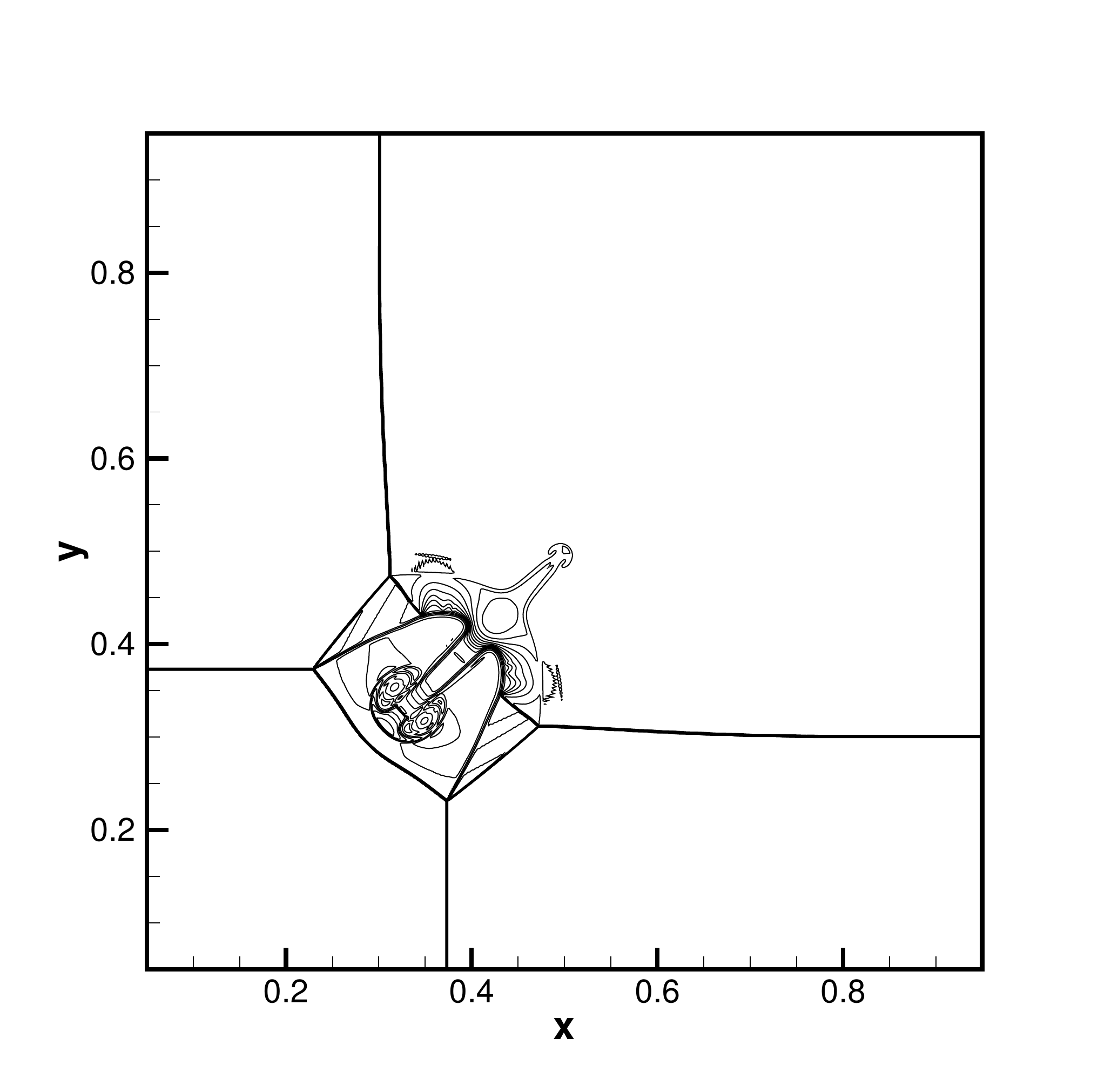}
\includegraphics[width=0.475\textwidth]{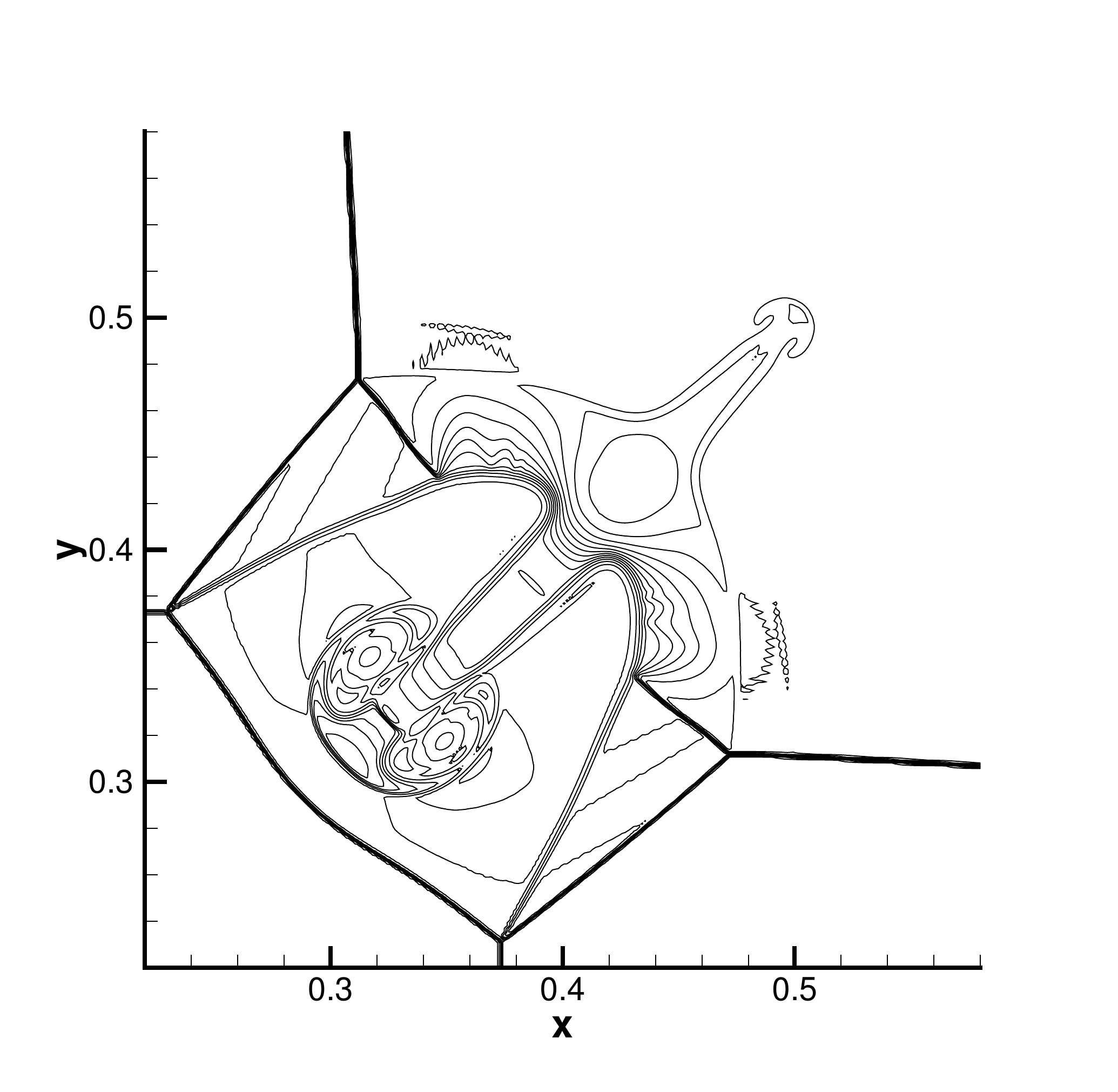}
\caption{\label{riemann-1} Two-dimensional Riemann problem: the density distributions and local enlargements for
 the first case  at $t = 0.3$ with $400\times 400$ (top) and $800 \times 800$ (bottom) uniform SVs.}
\end{figure}

\begin{figure}[!h]
\centering
\includegraphics[width=0.475\textwidth]{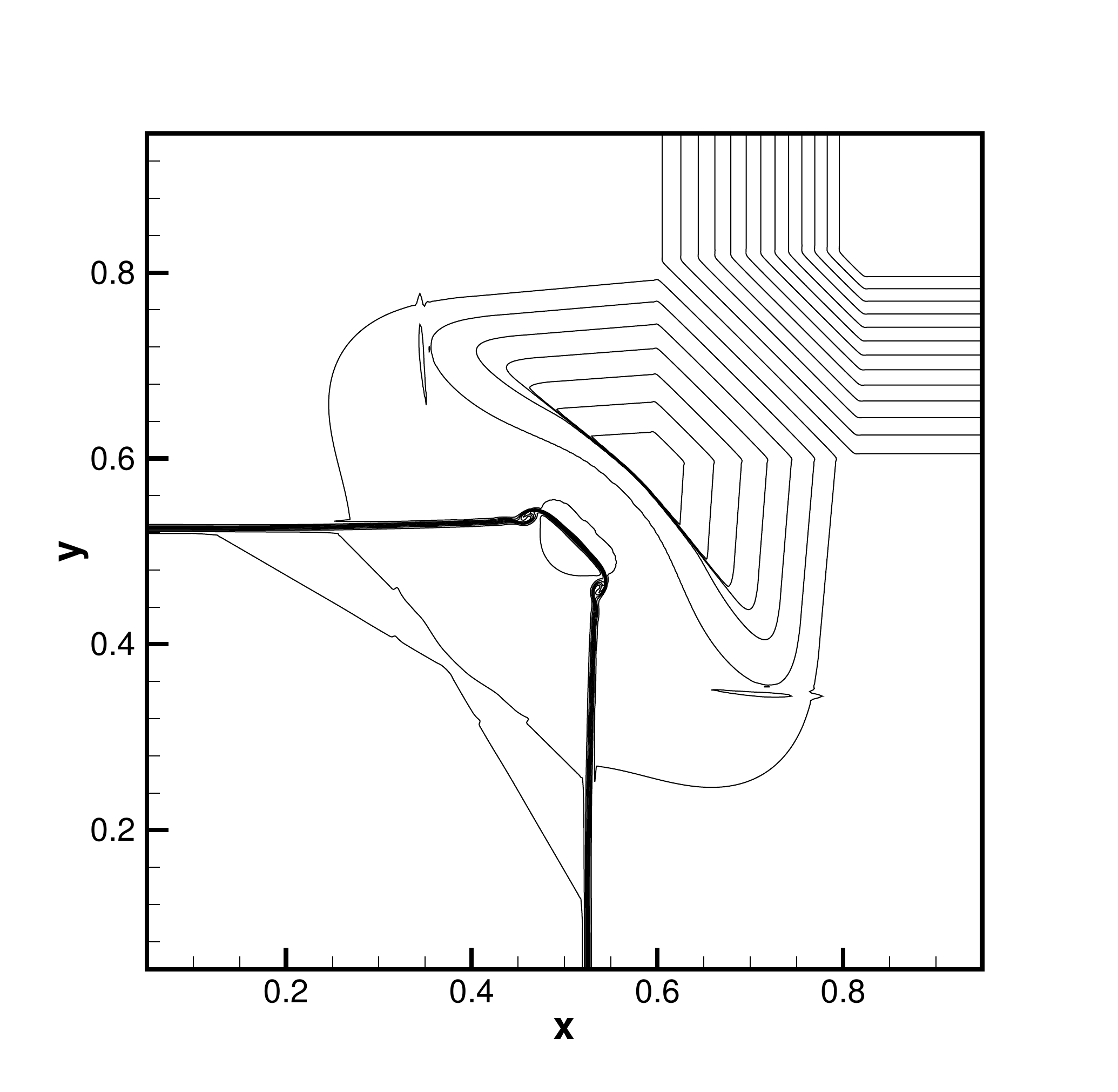}
\includegraphics[width=0.475\textwidth]{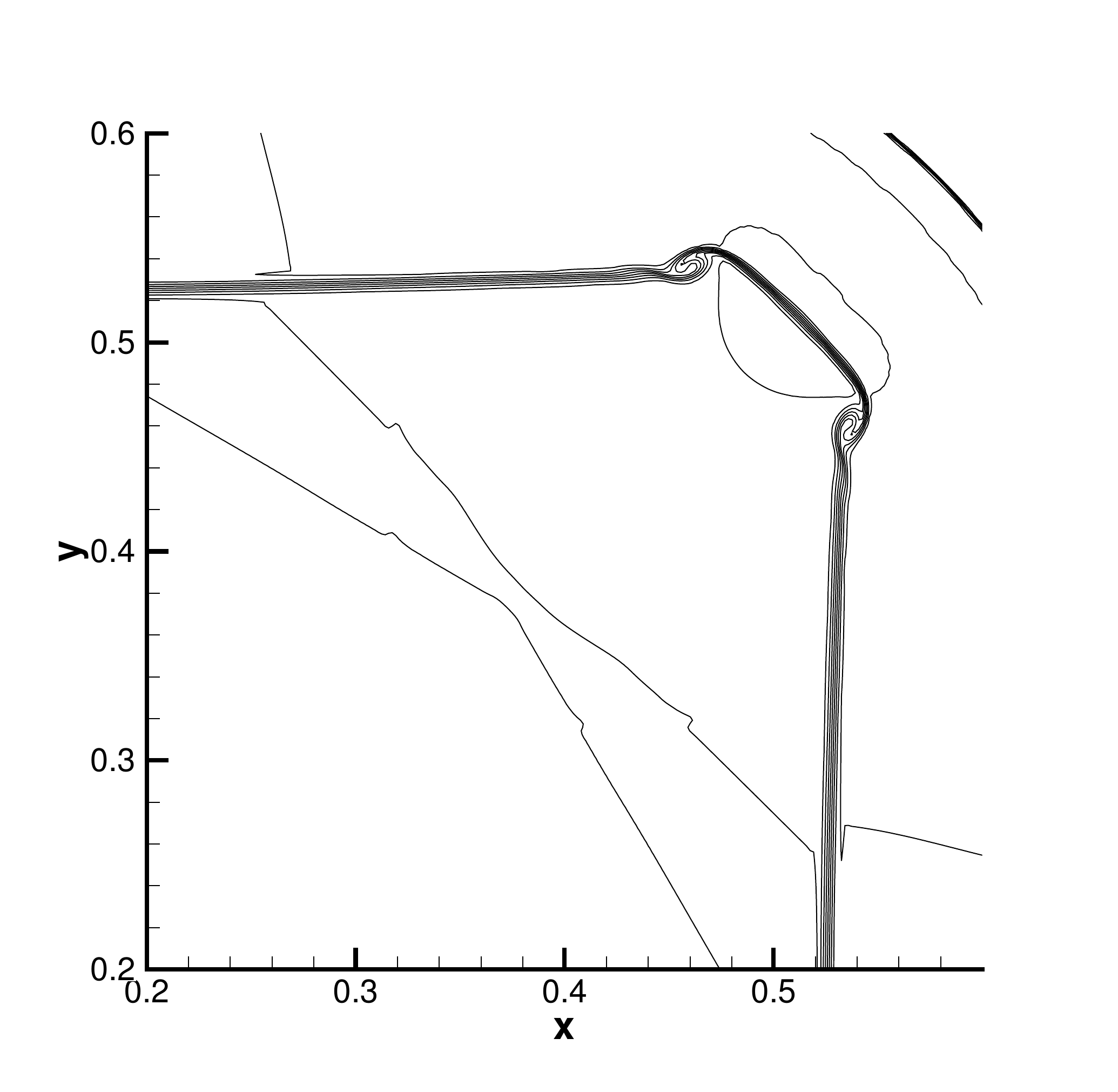}
\includegraphics[width=0.475\textwidth]{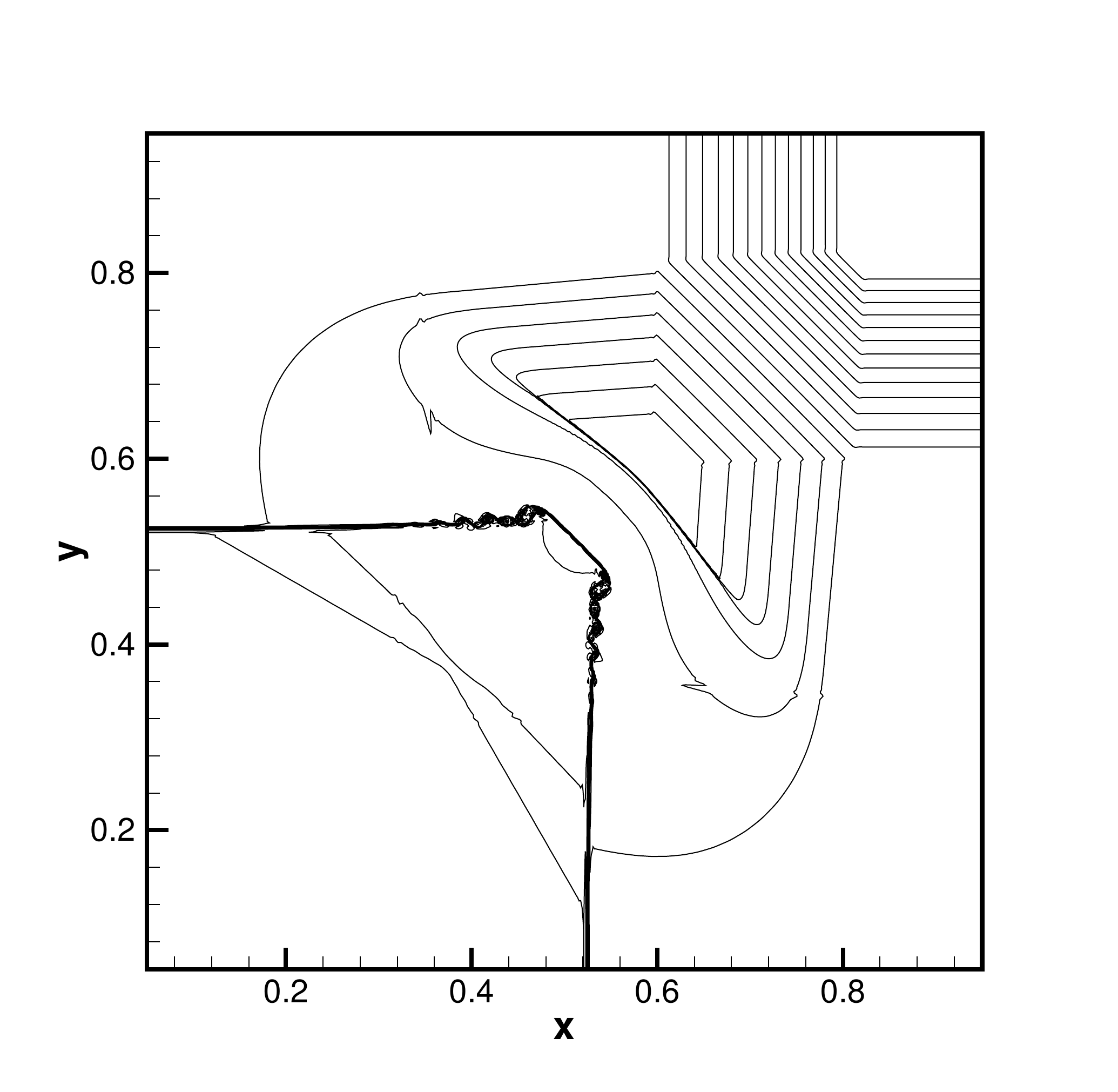}
\includegraphics[width=0.475\textwidth]{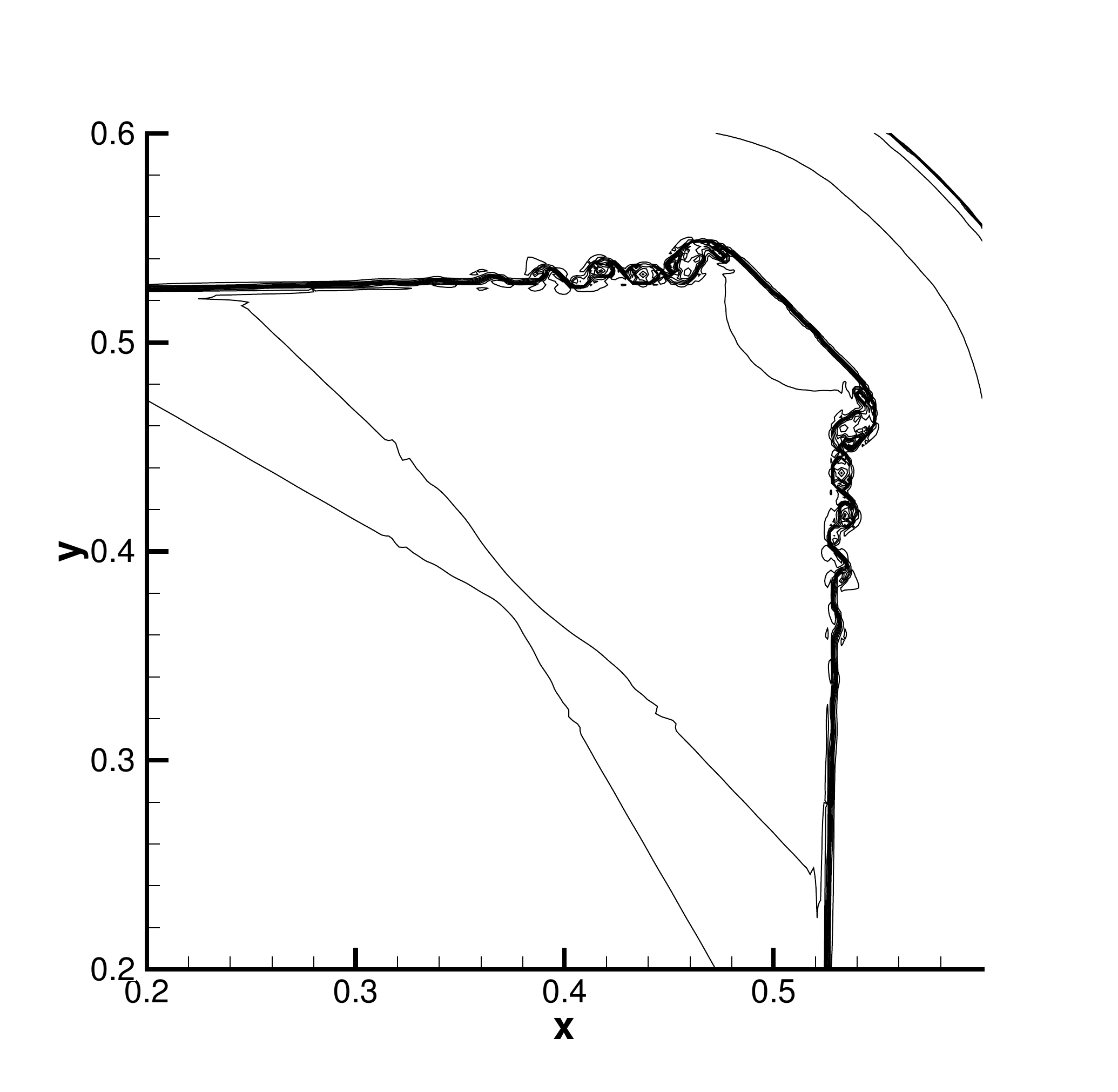}
\caption{\label{riemann-2} Two-dimensional Riemann problem:  the density distributions and local enlargements for the second case
at $t = 0.25$ with $300\times 300$ (top) and $600\times 600$ (bottom) uniform SVs.}
\end{figure}

The second case is to simulate the interaction between the interaction of two contacts with two rarefaction waves, and the initial conditions are
given as
\begin{align*}
(\rho, U, V, p)= \begin{cases}
(1,0.1,0.1,1), & x>0.5, y>0.5,\\
(0.5197,-0.6259,0.1,0.4), & x<0.5, y>0.5,\\
(0.8,0.1,0.1,0.4), & x<0.5, y<0.5,\\
(0.5197,0.1,-0.6259,0.4), & x>0.5, y<0.5.
\end{cases}
\end{align*}
The density distributions and the local enlargements at $t=0.25$  are
given in Figure.\ref{riemann-2} with  $300 \times 300$ and $600
\times 600$ uniform SVs. The results validate the good behavior of
the current method by which the roll-up is well captured.

\begin{figure}[!h]
\centering
\includegraphics[width=0.49\textwidth]{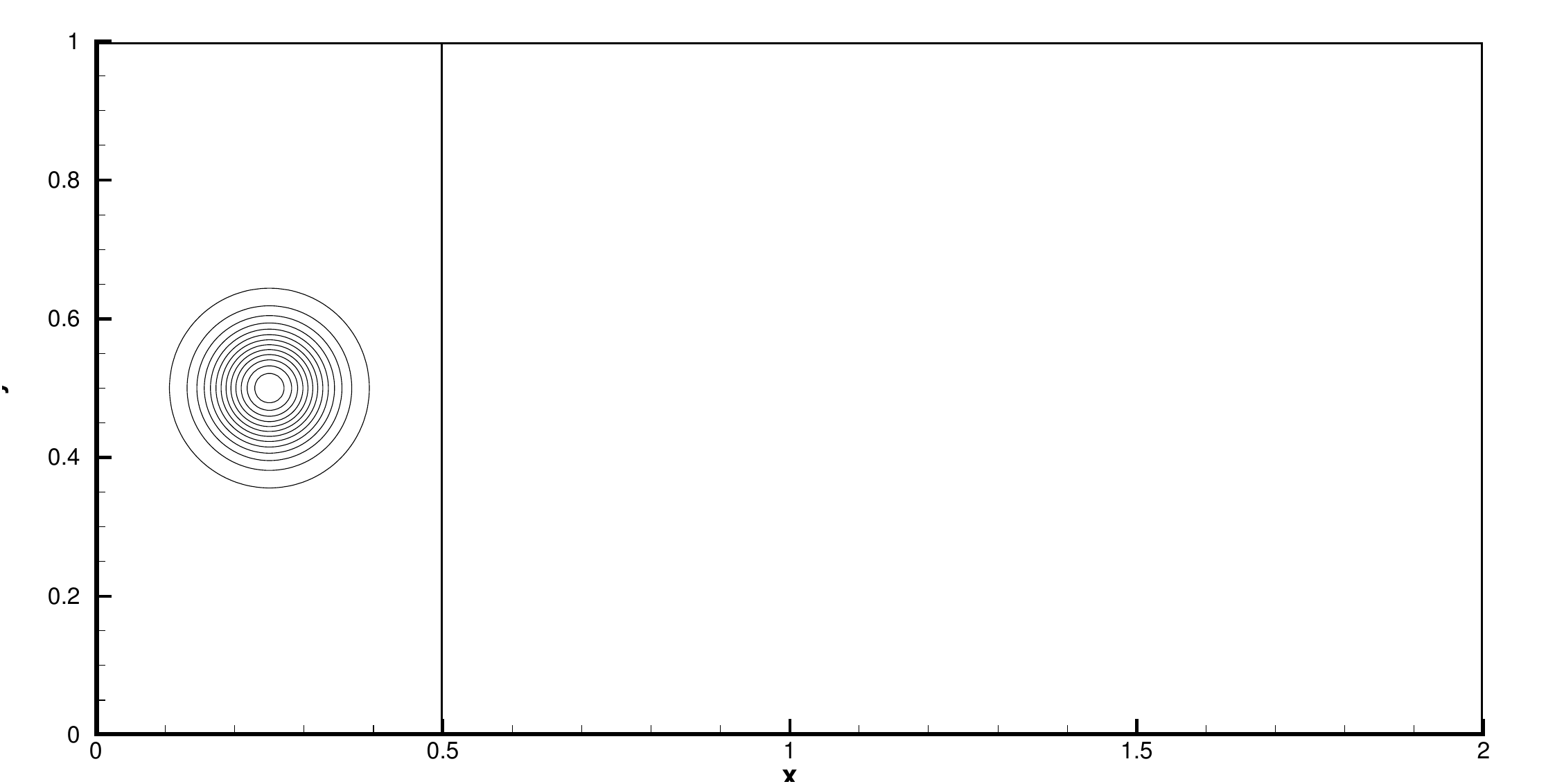}
\includegraphics[width=0.49\textwidth]{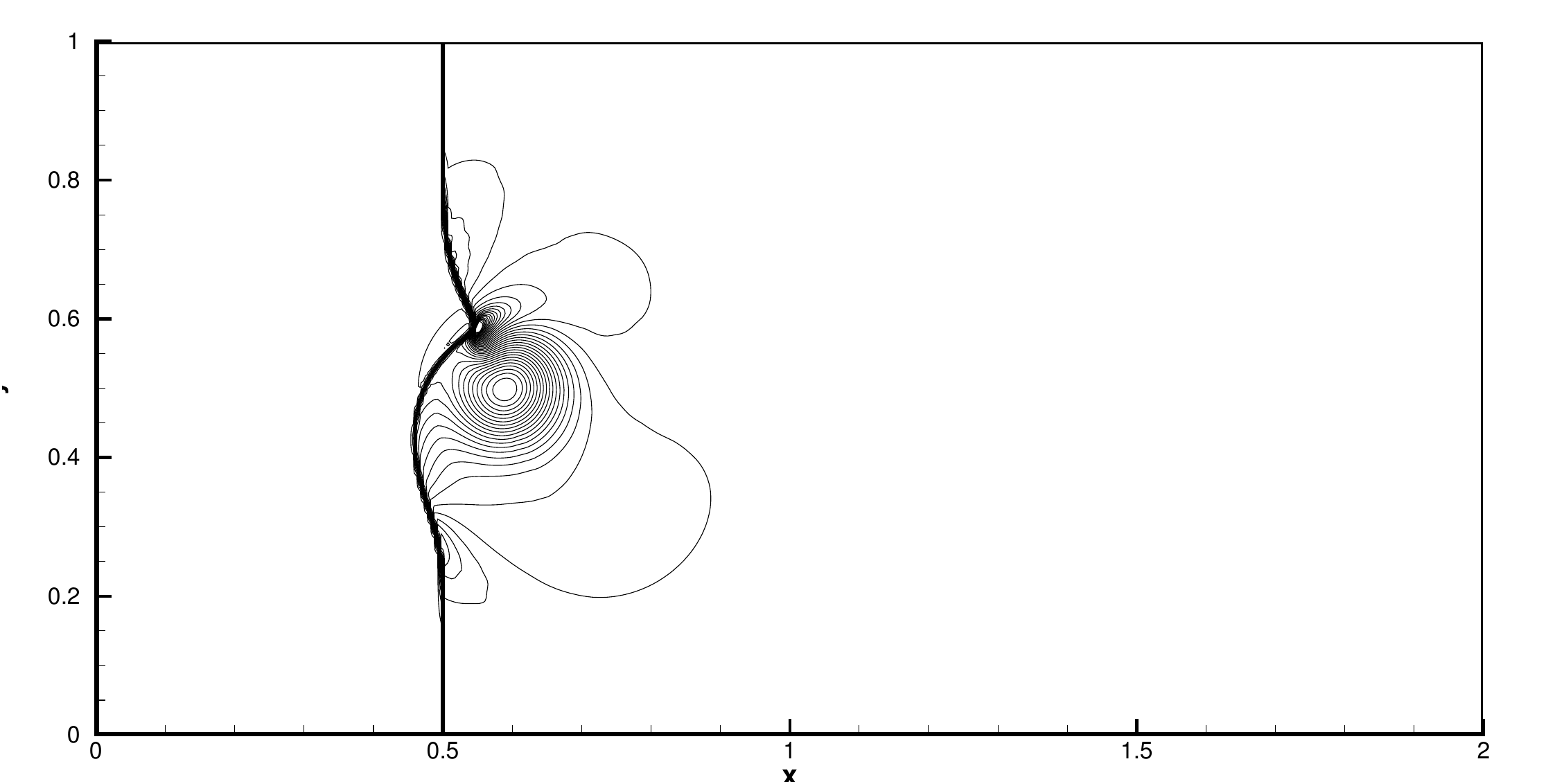}
\includegraphics[width=0.49\textwidth]{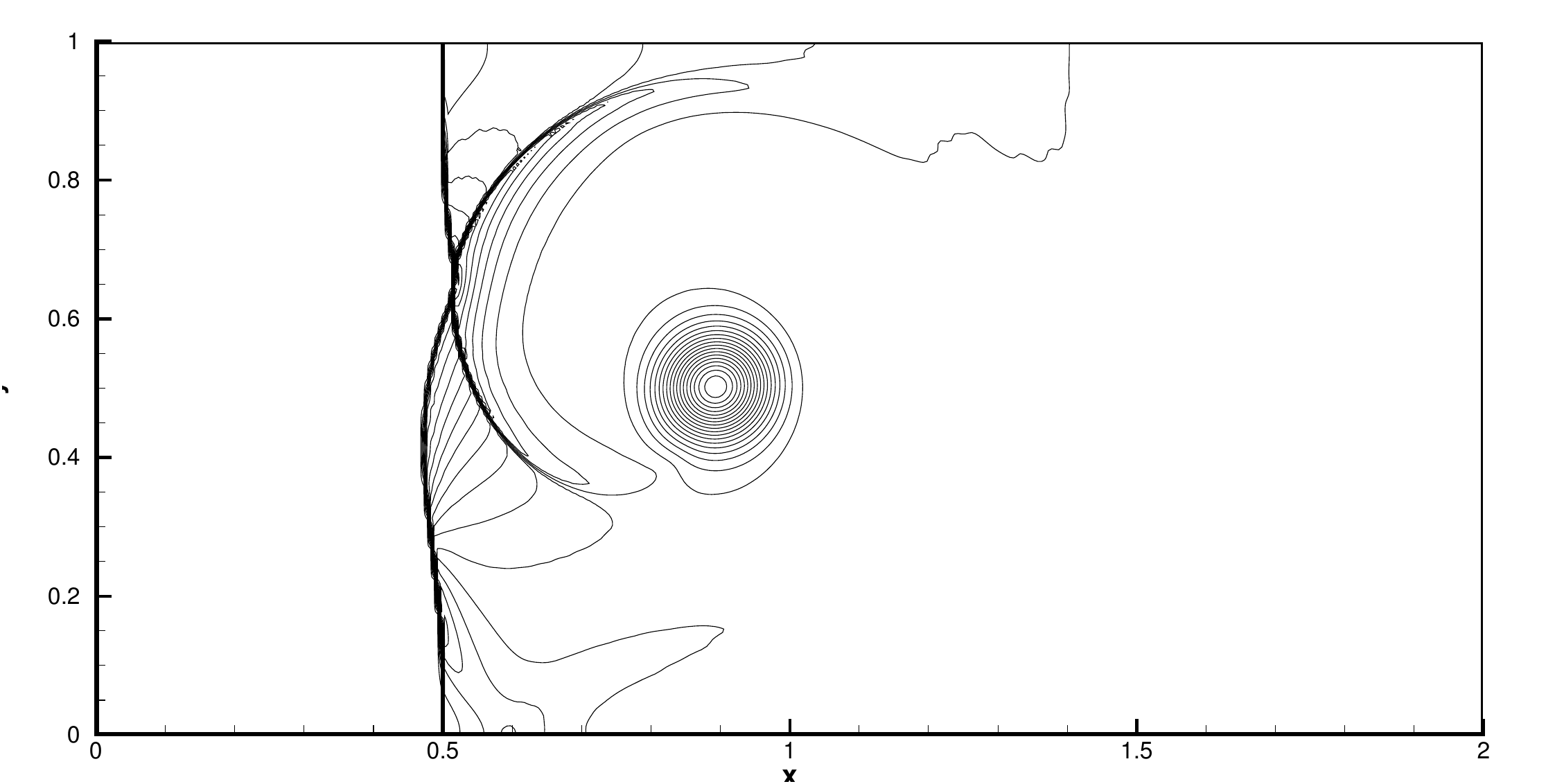}
\includegraphics[width=0.49\textwidth]{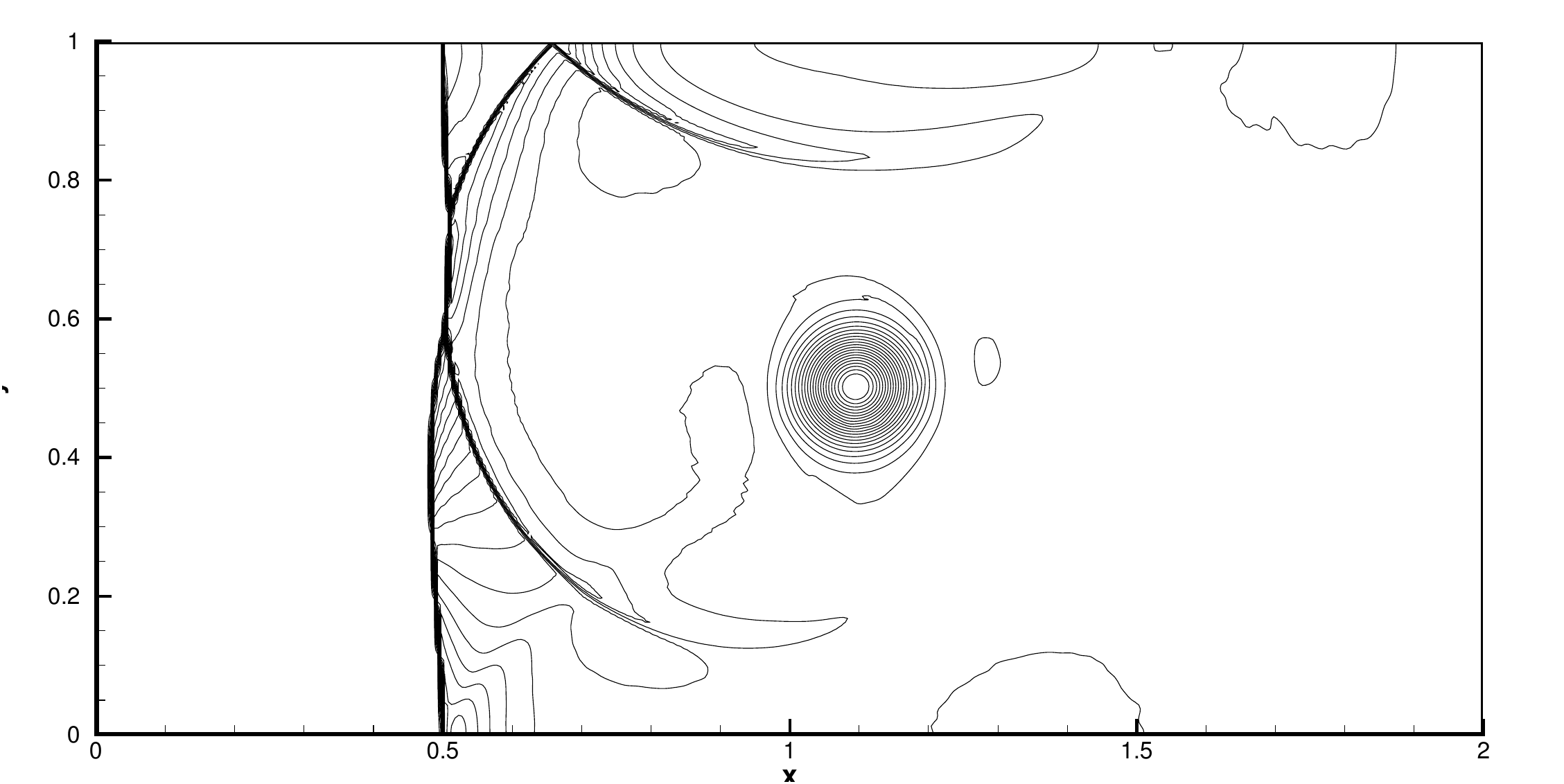}
\caption{\label{shock-vortex-1} Shock vortex interaction: the pressure distributions at $t = 0, 0.3, 0.6$ and $0.8$ with $400\times 200$ cells.}
\centering
\includegraphics[width=0.6\textwidth]{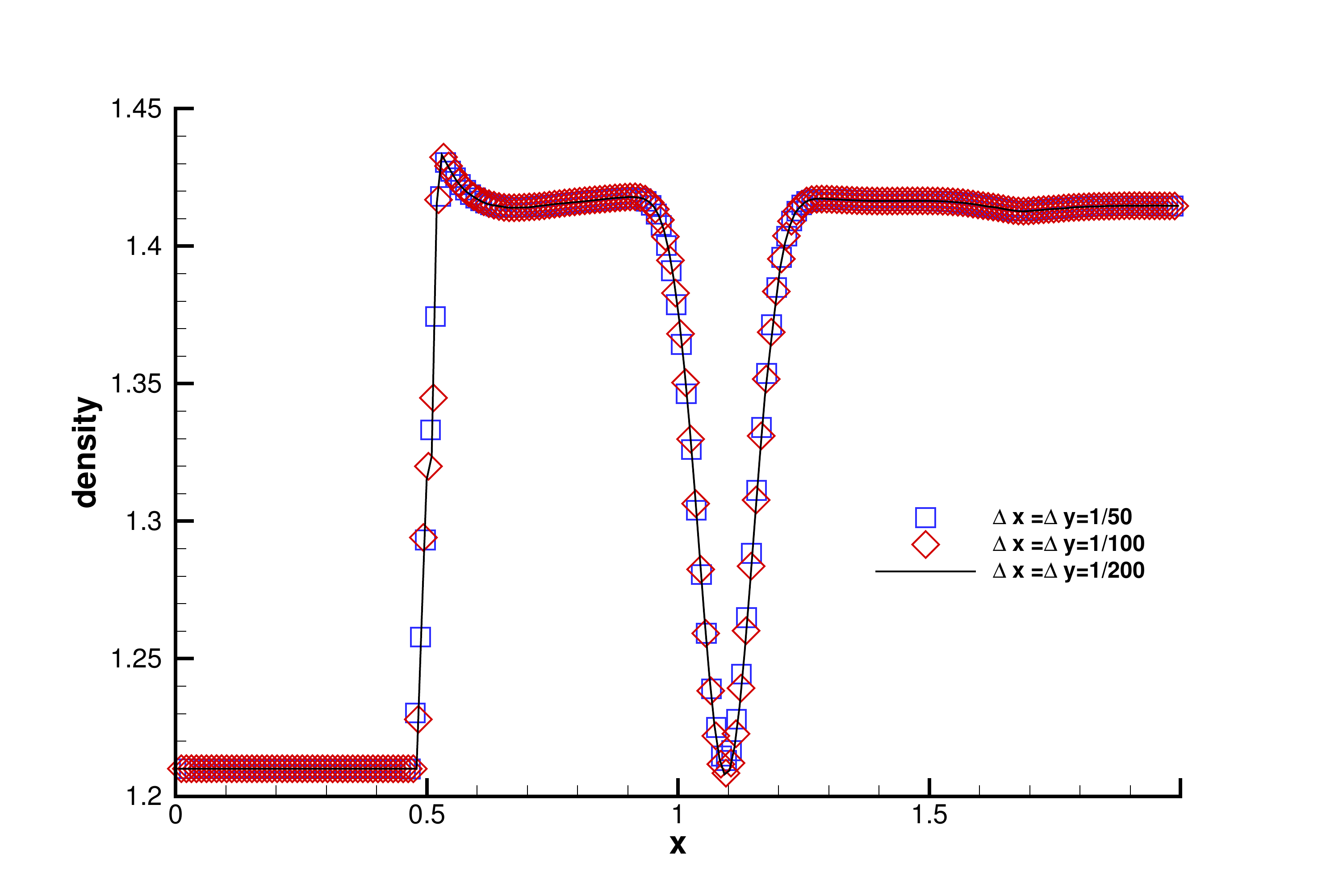}
\caption{\label{shock-vortex-2} Shock vortex interaction: the density distributions at $t = 0.8$ along the horizontal 
symmetric line $y = 0.5$ with mesh size $\Delta x = \Delta  y = 1/50,1/100,1/200$.}
\end{figure}

\subsection{Shock-vortex interaction}
The interaction between a stationary shock and a vortex for the
inviscid flow \cite{WENO-JS} is presented. The computational domain
is taken to be $[0,2] \times[0,1]$. A stationary Mach $1.1$ shock is
positioned at $x=0.5$ and vertical to the $x$-axis. The left
upstream state is $(\rho, U, V, p)=\left(M a^2, \sqrt{\gamma},
0,1\right)$, where $M a$ is the Mach number. A slight vortex
perturbation centered at $\left(x_c, y_c\right)=(0.25,0.5)$ is added
to the mean flow with the velocity $(U,V)$, temperature $T=p /
\rho$, and entropy $S=\ln \left(p / \rho^\gamma\right)$, expressed
as
\begin{equation*}
(\delta U, \delta V)=\kappa \eta e^{\mu\left(1-\eta^2\right)}(\sin \theta,-\cos \theta),
\end{equation*}
and
\begin{equation*}
\delta T=-\frac{(\gamma-1) \kappa^2}{4 \mu \gamma} e^{2 \mu\left(1-\eta^2\right)}, \delta S=0,
\end{equation*}
where $\eta=r / r_c$,
$r=\sqrt{\left(x-x_c\right)^2+\left(y-y_c\right)^2}$, $\kappa$
implies the strength of the vortex, $\mu$ controls the decay rate of
the vortex, and $r_c$ is the critical radius of the vortex with
maximum strength. In the computation,  these parameters are taken as
$\kappa=0.3, \mu=0.204$, and $r_c=0.05$. The reflecting boundary
conditions are used on the top and bottom boundaries. The inflow and
outflow boundary conditions are used for the left and right
boundaries. This case is tested by the uniform SVs with $\Delta
x=\Delta y=1 / 50,1 / 100$ and $1 / 200$. The pressure distributions
with $\Delta x=\Delta y=1 / 200$ at $t=0, 0.3, 0.6$ and $0.8$ are
shown in Figure.\ref{shock-vortex-1}. By $t=0.8$, one branch of the
shock bifurcations has reflected on the upper boundary and this
reflection is well captured. The detailed density distributions
along the center horizontal line with mesh size $\Delta x=\Delta
y=1/50,1/100$ and $1/ 200$ at $t=0.8$ presented in
Figure.\ref{shock-vortex-2} are convergent. Satisfactory results are
obtained and the accuracy of the scheme is well demonstrated.

 \begin{figure}[!h]
\centering
\includegraphics[width=0.95\textwidth]{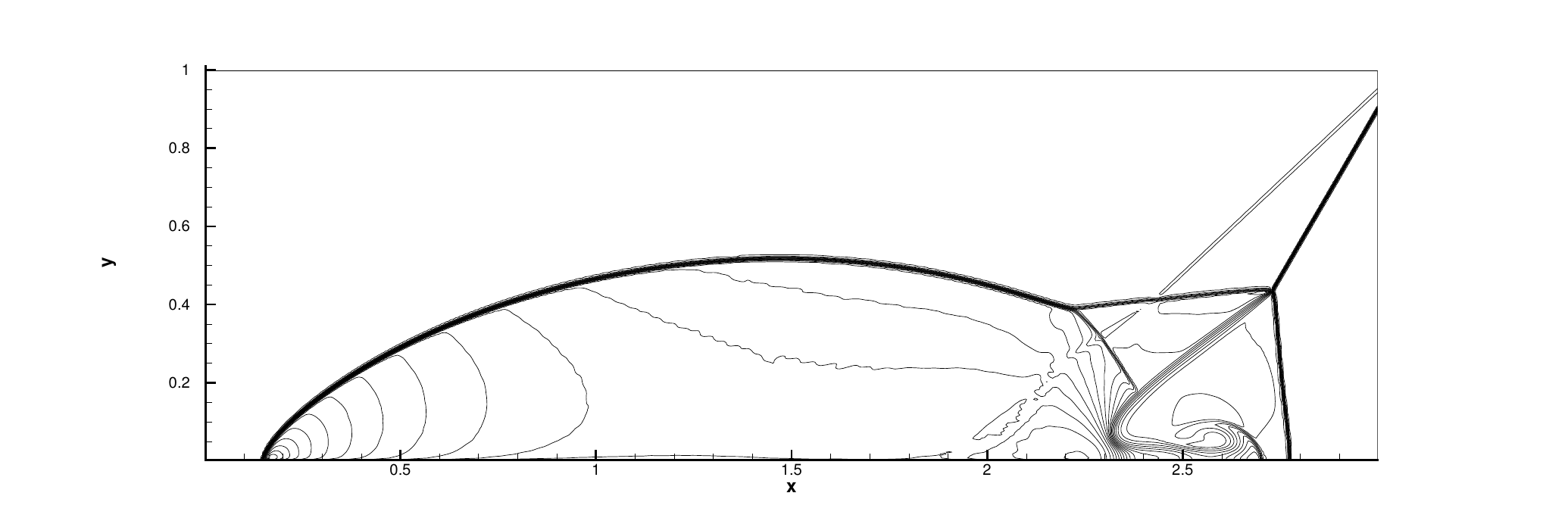}
\includegraphics[width=0.95\textwidth]{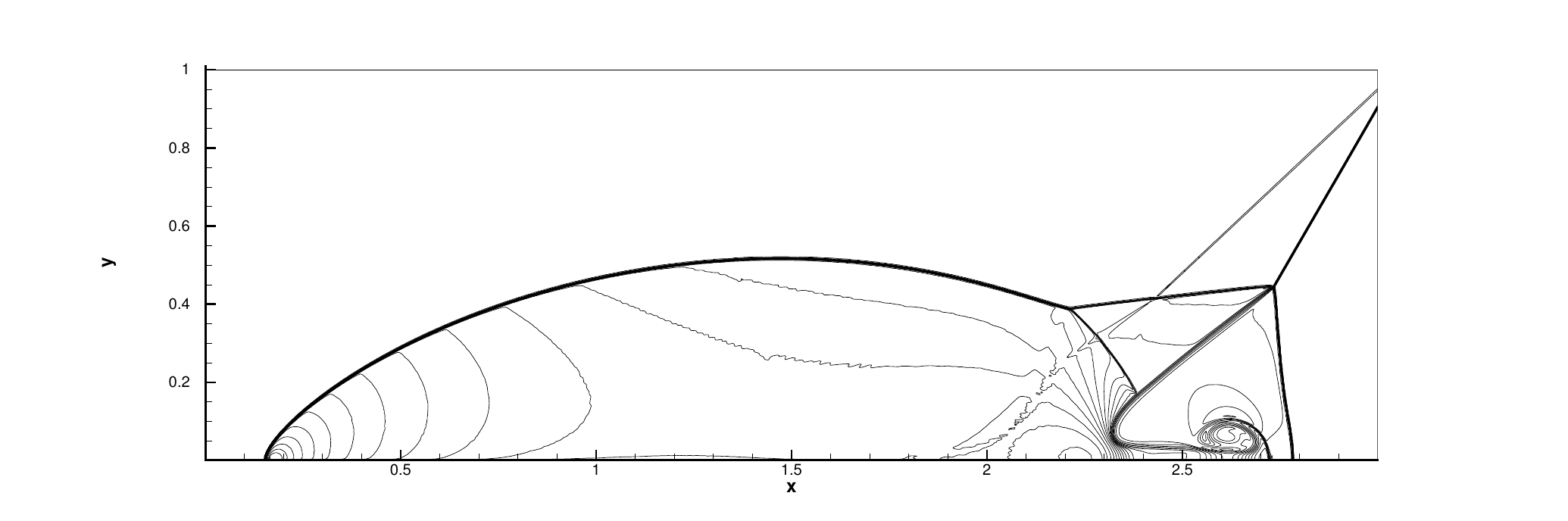}
\caption{\label{double-mach-1} Double Mach reflections: the density distributions at $t=0.2$ with 
$800 \times 200$ (top) and $1600 \times 400$ (bottom) uniform SVs.}
\end{figure}

\subsection{Double Mach reflection}
This problem was first proposed by Woodward and Colella
\cite{Case-Woodward} for the inviscid flow. The computational domain
is $[0,4] \times[0,1]$ and a reflecting solid wall lies along the
bottom of the computational domain starting from $x = 1/6$.
Initially a right-moving Mach 10 shock  makes a $60^{\circ}$ angle
with the reflecting wall starting at $(x,y)=(1/6,0)$ towards the top
of the computational domain. The initial pre-shock and post-shock
conditions are
\begin{equation*}
(\rho, U, V, p)=
\begin{cases}\left(8,4.125 \sqrt{3},-4.125,116.5\right), & x<\frac{1}{6}+\frac{1}{\sqrt{3}}y, \\
\left(1.4,0,0,1\right), & x>\frac{1}{6}+\frac{1}{\sqrt{3}}y.
\end{cases}
\end{equation*}
The inflow and outflow boundary conditions are adopted for the left
and right boundaries, respectively. The reflecting boundary
condition is used at the solid wall, and the exact post-shock
condition is imposed for the rest of the bottom boundary. At the top
boundary, the flow variables are set to describe the exact motion of
the Mach 10 shock. The density distributions and their local
enlargement with $800\times 200$ and $1600\times 400$ uniform SVs
at $t=0.2$  are presented in Figure.\ref{double-mach-1} and
Figure.\ref{double-mach-2},  respectively. The flow structure under
the triple Mach stem can be resolved clearly. However, the shear
layer seems to be smeared by the damping term, and more delicate
design of damping term still needs to be investigated in the future.

\begin{figure}[!h]
\centering
\includegraphics[width=0.45\textwidth]{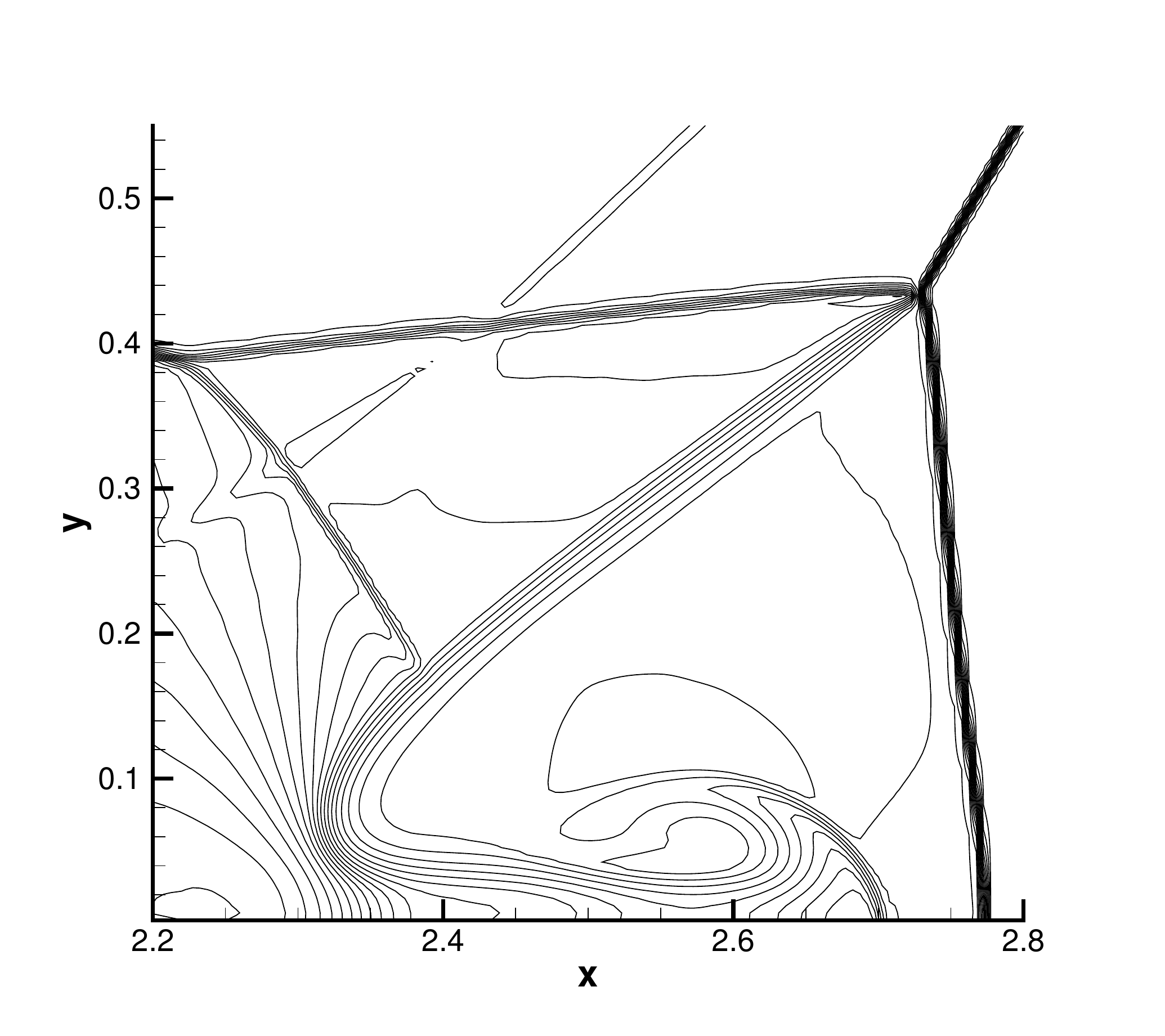}
\includegraphics[width=0.45\textwidth]{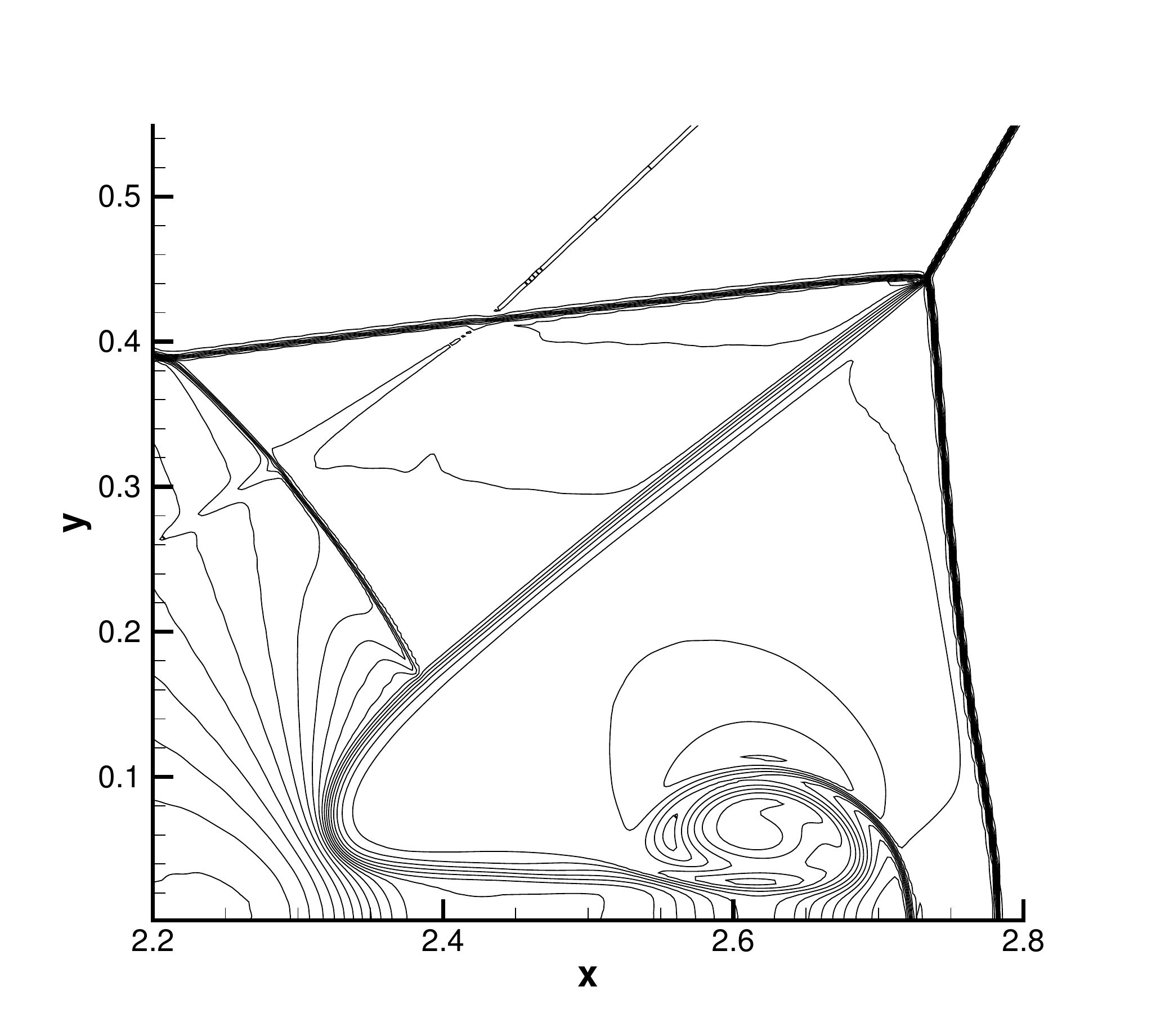}
\caption{\label{double-mach-2} Double Mach reflections: the local enlargement of density distributions  at $t=0.2$
with $800 \times 200$ (left) and $1600 \times 400$ (right) uniform SVs.}
\end{figure}

\section{Conclusion}
In this paper, an oscillation-free spectral volume method  is proposed for the systems of hyperbolic conservation laws.  
To suppress the oscillations near discontinuities,  a damping term is introduced to the standard spectral volume method. 
A mathematical  proof is provided to show that the proposed OFSV is stable and has optimal convergence rate  
and some excepted superconvergence properties when applied to linear scalar equations. Numerical experiments are presented to demonstrate 
the accuracy and capabilities of resolving discontinuities for the current scheme. 
The excepted order of accuracy of the current SV scheme is obtained,  
and the oscillations can be well controlled even for the problem with strong discontinuities.


\end{document}